\documentclass[11pt]{amsart}

\usepackage[a4paper,text={14.5cm,24cm},centering,headheight=15.5pt]{geometry}

\newcommand{\includecurve}[1]{\raisebox{-20pt}{\includegraphics[scale=.85]{figures/curves/#1}}}

\usepackage{mathtools}

\usepackage[utf8]{inputenc}
\usepackage[T1]{fontenc}
\usepackage[normalem]{ulem} 

\usepackage{amsmath,amsfonts,amssymb, amsthm}
\usepackage{bm} 

\usepackage[svgnames,x11names]{xcolor}

\usepackage{booktabs} 
\usepackage{tabularx}
\newcolumntype{L}{>{$}l<{$}}
\newcolumntype{R}{>{$}r<{$}}
\newcolumntype{C}{>{$}c<{$}}

\usepackage[shortlabels]{enumitem}
\setlist{itemsep=0.5pt,topsep=2pt}
\setlist[1]{wide}
\setlist[2]{wide,labelindent=2\parindent}
\setlist[3]{wide, labelindent=3\parindent}
\setlist[enumerate]{label=\rm{(\arabic*)}}
\setlist[enumerate,2]{label=\rm({\roman*}), }
\setlist[enumerate,3]{label=(\emph{\alph*})}
\setlist[itemize]{label=\raisebox{0.25ex}{\tiny$\bullet$}}

\usepackage{tikz}
\usetikzlibrary{cd,arrows,positioning,quotes}
\tikzcdset{arrow style=tikz}
\tikzcdset{row sep/normal=10pt}
\tikzset{link/.style={column sep=1.8cm,row sep=0.16cm}}
\tikzset{map/.style={row sep=0em, column sep=0em}}
\tikzset{bullet/.style={circle,fill=black,minimum size=5pt,inner sep=0pt}}
\tikzset{fuse/.style={pos=.7,fill=white, anchor=center, inner sep=3pt}}

\newcommand{\pcoor}[1]{%
  \begingroup\lccode`~=`: \lowercase{\endgroup
  \edef~}{\mathbin{\mathchar\the\mathcode`:}\nobreak}%
  \left[
  \begingroup
  \mathcode`:=\string"8000
  #1%
  \endgroup
  \right]
} 

\newcommand{\smallmat}[1]{\left(\begin{smallmatrix} #1 \end{smallmatrix}\right)}

\newcommand{\rat}{\dashrightarrow}

\newcommand{\iso}{\stackrel{\sim}{\to}}

\newcommand{\A}{\mathbb{A}}
\newcommand{\F}{\mathbb{F}}
\renewcommand{\P}{\mathbb{P}}
\newcommand{\p}{\mathbb{P}}
\newcommand{\G}{\mathbb{G}}

\newcommand{\Z}{\mathbf{Z}}
\newcommand{\N}{\mathbf{N}}

\newcommand{\Pl}{\mathcal P}

\DeclareMathOperator{\PGL}{PGL}

\DeclareMathOperator{\GL}{GL}
\DeclareMathOperator{\SL}{SL}

\newcommand{\lb}{\mathopen{\langle}}
\newcommand{\rb}{\mathclose{\rangle}}
\newcommand{\lbb}{\mathopen{\lb\!\lb}}
\newcommand{\rbb}{\mathclose{\rb\!\rb}}
\renewcommand{\setminus}{\smallsetminus}

\renewcommand{\epsilon}{\varepsilon}
\renewcommand{\phi}{\varphi}
\renewcommand{\le}{\leqslant}
\renewcommand{\ge}{\geqslant}
\renewcommand{\leq}{\leqslant}
\renewcommand{\geq}{\geqslant}

\newcommand{\id}{\textup{id}}

\newcommand{\ls}[1]{\lvert #1 \rvert} 

\let\div\relax
\DeclareMathOperator{\div}{div}

\DeclareMathOperator{\Aut}{Aut}
\DeclareMathOperator{\Bir}{Bir}

\DeclareMathOperator{\ord}{ord}

\DeclareMathOperator{\Jonq}{Jonq}

\numberwithin{equation}{section}

\usepackage{zref-clever}
\zcsetup{cap=true,hyperref=true,nameinlink=false,reffont=\upshape}
\newcommand{\cref}[1]{\zcref{#1}}

\newtheorem{theoremA}{Theorem}

\zcRefTypeSetup{theoremA}{Name-sg = Theorem}
\zcRefTypeSetup{theoremA}{Name-pl = Theorems}

\AddToHook{env/propositionA/begin}{\zcsetup{countertype={theoremA=proposition}}}

\newtheorem{theorem}{Theorem}[section]
\newtheorem{corollary}[theorem]{Corollary}
\newtheorem{proposition}[theorem]{Proposition}
\newtheorem{propdef}[theorem]{Proposition and Definition}
\newtheorem{lemma}[theorem]{Lemma}

\theoremstyle{definition}
\newtheorem{definition}[theorem]{Definition}

\newtheorem{example}[theorem]{Example}

\newtheorem{remark}[theorem]{Remark}

\AddToHook{env/corollary/begin}{\zcsetup{countertype={theorem=corollary}}}
\AddToHook{env/proposition/begin}{\zcsetup{countertype={theorem=proposition}}}
\AddToHook{env/lemma/begin}{\zcsetup{countertype={theorem=lemma}}}
\AddToHook{env/example/begin}{\zcsetup{countertype={theorem=example}}}
\AddToHook{env/remark/begin}{\zcsetup{countertype={theorem=remark}}}
\AddToHook{env/definition/begin}{\zcsetup{countertype={theorem=definition}}}
\AddToHook{env/setup/begin}{\zcsetup{countertype={theorem=setup}}}
\zcRefTypeSetup{setup}{
Name-sg = Setup,
name-sg = setup,
Name-pl = Setups,
name-pl = setups,
}

\makeatletter
\def\@cite#1#2{\textup{[{#1\if@tempswa , #2\fi}]}}
\makeatother

\newcommand{\parent}[1]{\textup{(}#1\textup{)}} 

\newcommand{\nicecolor}{Navy}
\usepackage[
  backref=page,
  colorlinks=true,
  allcolors=\nicecolor
]{hyperref}
\usepackage[all]{hypcap} 
\usepackage[depth=2, open, openlevel=0, numbered]{bookmark} 
\usepackage{upref} 

\renewcommand{\k}{\mathbf{k}}

\title{Plane Cremona maps preserving a Euclidean volume form}
\date{\today}

\author{Marc Abboud}
\address{Marc Abboud, Universit\'e de Neuch\^atel,
Rue Emile-Argand 11,
CH-2000 Neuch\^atel,
Switzerland}
\email{marc.abboud@unine.ch}

\author{J\'er\'emy Blanc}
\address{J\'er\'emy Blanc, Universit\'e de Neuch\^atel,
Rue Emile-Argand 11,
CH-2000 Neuch\^atel,
Switzerland}
\email{jeremy.blanc@unine.ch}

\author{St\'ephane Lamy}
\address{St\'ephane Lamy, Institut de Math\'ematiques de Toulouse,
Universit\'e de Toulouse,
118 route de Narbonne,
F-31062 Toulouse Cedex 9,
France}
\email{slamy@math.univ-toulouse.fr}

\begin{document}

\begin{abstract}
We give a presentation by generators and relations of the group of birational transformations of the plane preserving a differential form with poles along a triple line.
This group strictly contains the automorphism group of the  affine plane.
\end{abstract}

\maketitle

\tableofcontents

\section{Introduction}
In birational geometry, one often studies \emph{Calabi--Yau pairs}. These are couples $(X,D)$, where $X$ is a projective variety and $D$ is an effective divisor corresponding to the poles of a canonical divisor on $X$. One then studies volume preserving birational maps between two such pairs; these are maps compatible with the associated volume forms.
We denote by $\Bir(X,D)$ the group of volume preserving birational self-maps of $X$.

If $X=\P^2$, there are, up to birational conjugacy, three possible cases for $D$: either $D$ is a smooth cubic curve, or $D$ is the union of the three coordinate lines, or $D$ is a triple line (this follows from \cite{DillerLin}, see also \cref{p:birational_classification}).

When $D$ is a smooth cubic curve, the group $\Bir(\P^2, D)$ coincides with the group of maps that restrict to automorphisms on $D$, also called the decomposition group of $D$.
In this case, by \cite[Theorem 1.4]{Pan}, the group $\Bir(\P^2, D)$ is generated by its elements of degree $2$.

When $D$ is the union of the three coordinate lines, the group $\Bir(\P^2, D)$ contains the toric Cremona group $\GL_2(\Z) \rtimes (\G_m)^2$. This is the group of automorphisms of the torus $\P^2\setminus D$, where $\GL_2(\Z)$ acts by monomial maps, and the torus $(\G_m)^2$ acts diagonally.
It was noticed in \cite{Usnich} that $\Bir(\P^2, D)$ contains nontoric elements, such as the map $P\colon (x,y) \mapsto (x,(y+1)/x)$, and then it was proved in \cite{Blanc2013} that $\Bir(\P^2, D)$ is generated by $P$ and the toric maps.
The same paper describes relations between these generators.

In the present work, we treat the case where $D = 3L$ is a triple line.
Denoting by $\A^2$ the affine plane obtained by taking the complement of $L$ inside $\P^2$, a natural subgroup of $\Bir(\P^2, 3L)$ is given by the group $\Aut(\A^2)$ of polynomial automorphisms.
Observe that in this case, we can choose some coordinates $x,y$ on $\A^2$ and write $\omega = dx \wedge dy$, and the constant $\lambda$ in the formula $f^* \omega = \lambda \omega$ is the Jacobian determinant of the polynomial automorphism $f$.
However, $\Bir(\P^2, 3L)$ contains more maps, such as the maps $(x,y)\mapsto (x+\frac{1}{y},y)$ and $(x,y) \mapsto (-y^2x + y, \frac1y)$.
By the Jung--van der Kulk theorem \cite{Jung,Kulk,CremonaBook}, the group $\Aut(\A^2)$ is the amalgamated product of the affine subgroup and the triangular subgroup along their intersection.
The triangular group being the group of polynomial automorphisms preserving the pencil of lines $y = $ constant, it is natural to consider the subgroup $B \subset \Bir(\P^2, 3L)$ of Jonquières maps preserving this same pencil.
Our main result is then:

\begin{theoremA}
\label{t:main}
The group $\Bir(\P^2, 3L)$ admits the following presentation:
\[
\Bir(\P^2, 3L) = (A *_{A \cap B} B) / N,
\]
where $A$ is the affine subgroup, $B$ is the subgroup of volume preserving Jonquières maps preserving the pencil $y = $ constant, and $N$ is the normal closure in the abstract amalgamated product $A *_{A \cap B} B$ of the single element
\[
\left( \bigl(-xy^2+y, \tfrac1y \bigr) \circ (y,x) \circ \bigl(-x+ \tfrac1y,y \bigr) \circ (y,x) \right)^2\!\!.
\]
\end{theoremA}

Restricting to the Calabi--Yau pairs $(X,D)$ where $D$ is a reduced divisor, $X$ is a terminal variety, and $(X, D)$ has at most dlt singularities, Corti and Kaloghiros \cite{CortiKaloghiros} proved a version of the Sarkisov program for factorising volume preserving maps between Calabi--Yau pairs with a structure of Mori fibre space.
This framework was then used in \cite{AraujoCortiMassarenti} to study the volume preserving map of $(\P^3, Q)$, where $Q$ is a smooth quartic surface in $\P^3$.
Similarly, the result of \cite{Blanc2013} mentioned above could be seen as a consequence of the general Sarkisov factorisation applied to the case of the dlt pair $(\P^2, L_1 + L_2 + L_3)$.
By contrast, the case of the pair $(\P^2, 3L)$ does not fit in this framework, because the boundary divisor is not reduced.

We now describe the way we prove Theorem~\ref{t:main} and how the article is organised.

Section~\ref{s:generalities} contains some classical results and preliminaries on birational maps of the plane, linear systems, base points and volume preserving maps.
In \S\ref{SubSec:Links} we also introduce some basic results on Jonqui\`eres maps associated to volume forms.

Section~\ref{sec:generation} proves the generation part of Theorem~\ref{t:main}.
To prove that $\Bir(\P^2, 3L)$ is generated by $A$ and $B$, actually we prove that $\Bir(\P^2, 3L)$ is generated by the set $\Jonq(3L)$ of  Jonquières maps preserving the form $\omega$ with triple pole along $L$.
The sets $A\cup B$ and $\Jonq(3L)$ generate the same subgroup of $\Bir(\P^2)$, as $\Jonq(3L)=\{ aba'\mid a,a'\in A,b\in B\}$.
Translating the relation of \cref{t:main} in this setting (which becomes shorter), we prove the following reformulation of \cref{t:main}:

\begin{theoremA}
\label{t:main_bis}
The group $\Bir(\P^2, 3L)$ admits the presentation
\[
\Bir(\P^2, 3L) = \lb \Jonq(3L) \mid R \rb,
\]
where $R$ consists of all relations of length $\le 3$, and the following single extra relation of length $4$:
\[
\left( \bigl(-xy^2+y, \tfrac1y \bigr) \circ \bigl(x, -y+ \tfrac1x\bigr) \right)^2.
\]
\end{theoremA}

Section~\ref{s:volume preserving-jonquieres-maps} studies the elements of $\Jonq(3L)$. A description in coordinates is given in~\cref{l:jonq_preserving}, and allows to prove that the abelianisation of $\Jonq(3L)$ is $\k^*$ (\cref{prop:abelianisation-trivial}).
Then, \cref{Prop:BasePtsJonq}  gives a necessarily and sufficient conditions for a set of points to be base points of an element of $\Jonq(3L)$.
We also introduce simple maps in \S\ref{Sec:SimpleMaps} and describe relations associated to them, showing that all of them hold in the group $ \lb \Jonq(3L) \mid R \rb$ of \cref{t:main_bis}.

Section~\ref{Sec:ProofBis} introduces peak configurations and degree lowering relations, and gives the proof of \cref{t:main_bis}.

Section~\ref{sec:reformulation} explains how  \cref{t:main} follows from \cref{t:main_bis}.

\subsection*{Acknowledgments}
This article was written while S.~Lamy was on a ``d\'el\'egation CNRS'' for a one-year stay at the University of Neuch\^atel, supported by the Labex CIMI.
The three authors acknowledge support by the Swiss National Science Foundation Grant ``Birational transformations of higher dimensional varieties'' $200020\_214999$.
\subsection*{AI Statement}
  The authors did not use any AI in any sense for the article.

\section{Generalities}\label{s:generalities}

We work over an algebraically closed field $\k$ of arbitrary characteristic.

\subsection{Homaloidal linear systems}

Given $f \in \Bir(\P^2)$, the linear system $\Lambda$ given by inverse images of lines is called the \emph{homaloidal linear system} associated with $f$.
Writing $f = \pcoor{P_0:P_1:P_2}$ with $P_i \in \k[x,y,z]$ homogeneous of the same degree $d$ and without a nonconstant common factor, $\Lambda$ is the system of curves with equation $a_0 P_0 + a_1 P_1 + a_2 P_2 = 0$.
Consider a minimal resolution of $f$ by blowups:
\[
\begin{tikzcd}
& Z \ar[ld,"\eta",swap] \ar[rd,"\eta'"]\\
\P^2 \ar[rr,"f",dashed] && \P^2.
\end{tikzcd}
\]
Here $\eta = \eta_1 \circ \dots \circ \eta_r$ is a sequence of blowups $\eta_i\colon X_{i+1} \to X_{i}$, with $X_{r+1} = Z$ and $X_1 = \P^2$.
Each point $p_i\in X_i$ blown up by $\eta_i$ is called a \emph{base point} of $f$.
We denote by $m_i$ the multiplicity at $p_i$ of a general member of the strict transform of $\Lambda$.
We say that $(d;m_0, \dots, m_r)$ is the homaloidal type of $f$.

Given two indices $i<j$, we say that
\begin{itemize}
\item
 $p_j$ is \emph{infinitely near to $p_i$} if $(\eta_{i}\circ \dots \circ \eta_j)(p_j)=p_i$.
 \item
$p_j$ is \emph{in the first neighbourhood of $p_i$}, denoted by $p_i \leftarrow p_j$, if $p_j$ is infinitely near to $p_i$ and not to any $p_{i'}$ with $i<i'<j$. This means that $\eta_{i+1}\circ \dots \circ \eta_j$ is a local isomorphism at $p_j$ and sends $p_j$ onto a point of the exceptional curve $\eta_i^{-1}(p_i)$ of $p_i$.
 \item
$p_j$ is \emph{proximate} to $p_i$ if $p_j$ lies on the strict transform of the exceptional curve $\eta_i^{-1}(p_i)$.
 \end{itemize}
Note that $p_j$ is infinitely near to $p_i$ if and only if there is a chain of points $p_i \leftarrow \dots \leftarrow p_j$.
Moreover, if $p_j$ is proximate to $p_i$, then all points in this chain except $p_i$ are proximate to $p_i$.

As the strict transform of exceptional curves always is a tree, each base point $p$ is proximate to $0$, $1$, or $2$ other base points.
We say that $p$ is a \emph{proper} base point if $p$ is not proximate to any other base point, and thus corresponds to a point of $\P^2$.

We also say that an irreducible curve of $\P^2$ passes through a point $p_i$ if its strict transform passes through $p_i$.

\begin{lemma}[Noether's relations]
\label{l:noether_rel}
\cite[2.5.1]{Carraminana}
\cite[2.34]{CremonaBook}
Let $(d;m_1, \dots, m_r)$ be a homaloidal type.
Then
\begin{align*}
\sum_{i=1}^r m_i = 3d - 3, &&
\sum_{i=1}^r m_i^2 = d^2 - 1.
\end{align*}
\end{lemma}

\begin{lemma}
\label{TwoptsHom}
Let $(d;m_1, \dots, m_r)$ be a homaloidal type, with $d\ge 1$. Then, for all $1\le i < j\le m_r$ we have \[d\ge m_i+m_j.\]
\end{lemma}

\begin{proof}
Let $\Lambda$ be a linear system having degree $d$ and multiplicities $m_1,\ldots,m_r$ at $p_1,\ldots,p_r$.
We use the following basic fact: if $p_{k} \leftarrow p_{k'}$, then the corresponding multiplicities satisfy $m_k \ge m_{k'}$.
We can then assume that the point $p_i$ is proper, and that either the point $p_j$ is proper and distinct from $p_i$, or $p_i \leftarrow p_j$.
In both cases, there is a line through the two points, and Bézout's theorem and the irreduciblity of a general element of $\Lambda$ gives $d\ge m_i+m_j$.
\end{proof}

A map $\phi \in \Bir(\P^2)$ of degree $d$ is called \emph{Jonquières} if either $d = 1$, or $d \ge 2$ and $\phi$ admits a base point $p_0$ of multiplicity $d-1$.
This implies that $\phi$ admits exactly $2d -2$ other base points, all of multiplicity one.
If $d \ge 3$, the point $p_0$ is proper and is called the maximal base point of $\phi$.
When $d = 2$, by convention any proper base point of $\phi$ will be called a maximal base point, there are up to three choices.
In all cases, $\phi$ sends the pencil of lines through a maximal base point $p_0$ to the pencil of lines through some point $q_0$, and $\phi^{-1}$ is a Jonquières map with maximal base point $q_0$.

Let $(d;m_0, \dots, m_r)$ be the homaloidal type of a map $f \in \Bir(\P^2)$, and assume the base points $p_i$ are ordered such that $m_0 \ge m_1 \ge \dots \ge m_r$.
We say that $p_i$ is a \emph{major base point} if $2m_i > d - m_0$.

\begin{proposition}
\label{p:simplifying_jonquieres}
\cite[8.3.3]{Carraminana}, \cite[12.24]{CremonaBook}.
Let $f \in \Bir(\P^2)$, with major base points $p_0, \dots, p_h$, and corresponding multiplicities $m_0 \ge \dots \ge m_h$.
Let $k = (h+2)/2$ if $h$ is even
\parent{resp. $k = (h+1)/2$ if $h$ is odd}.
Then $k \ge 2$, and there exists a Jonquières map $\phi$ of degree $k$ with maximal base points at $p_0$, and other base points $p_1, \dots, p_{2k-2}$.
Moreover, $\deg(f \circ \phi^{-1}) < \deg(f)$.
\end{proposition}

\begin{proposition}
\label{PropDelta}
Let $\phi\in \Bir(\P^2)$ be a Jonquières map of degree $r\ge 2$, with maximal base point $p_0$ of multiplicity $r-1$, and other
base points $p_1,\ldots,p_{2r-2}$ of multiplicity $1$.
Let $\Lambda$ be a homaloidal linear system of degree $d = \deg(\Lambda)$, and write $m_q = m_q(\Lambda)$ the multiplicity of $\Lambda$ at a point $q$.
Then
\begin{align*}
\deg(\Lambda)-\deg(\phi(\Lambda))
&= \sum_{i=1}^{2r-2} m_{p_i}-(d-m_{p_0})(r-1)\\
&=\sum_{i=1}^{r-1}(m_{p_0}+m_{p_{2i-1}}+m_{p_{2i}}-d).
\end{align*}
In particular, if $\deg(\phi(\Lambda))< \deg(\Lambda)$, there exists $i\in \{1,\dots,r-1\}$ such that $m_{p_0}+m_{p_{2i-1}}+m_{p_{2i}}>d$.
\end{proposition}

\begin{proof}
If $\phi \in \Bir(\P^2)$ is an arbitrary map with homaloidal type $(\delta; \mu_{p_0}, \mu_{p_1}, \dots)$, then we have (see e.g.~\cite[Lemma 8.16]{CremonaBook}):
\[
\deg \phi(\Lambda) = d \delta - \sum_{i\ge 0} m_{p_i} \mu_{p_i}.
\]
Specializing to the case of a Jonquières map of degree $r$, we have $\mu_{p_0}= r-1$ and $\mu_{p_i} = 1$ for $i \ge 1$, so that
\[
\deg \phi(\Lambda) = dr - m_{p_0}(r-1) - \sum_{i\ge1} m_{p_i}.
\]
Subtracting $\deg \Lambda = d$, we obtain the desired equality.
\end{proof}

\subsection{Volume preserving maps}

Let $X$ be a smooth rational projective surface.
An element $D \in \ls{-K_X}$ of the anticanonical linear system is called an \emph{anticanonical curve}.
Given a choice of such an anticanonical curve, we say that $(X,D)$ is a Calabi--Yau pair.
Let $\omega$ be a rational 2-form on $X$.
We say that $\omega$ is polar if it does not vanish.
In other words, the divisor $-\div(\omega)$ is effective, or equivalently is an anticanonical curve.

Conversely, to an anticanonical curve we can associate a polar 2-form $\omega$ such that $D = -\div(\omega)$, and $\omega$ is unique up to a multiplicative constant.
Indeed, let $\omega_0$ be an arbitrary $2$-form, and let $K_X$ denote the associated divisor.
By definition $D \in \ls{-K_X}$, so there exists a rational function $f$ such that $D = \div(f) + K_X$. Then $\omega = cf\cdot \omega_0$ for some constant $c$.

We now recall two results from \cite{Blanc2013}:
\begin{lemma}
\label{l:mult_blowup}
Let $\eta\colon Y \to X$ be the blowup of a point $p$ on a smooth projective surface $X$, and denote by $E\subset Y $ the exceptional divisor.
Let $K_X$ be the divisor associated to a form $\omega$ on $X$, and $K_Y$ the divisor associated to the form $\eta^*(\omega)$.
Denote by $\tilde K_X$ the strict transform of $K_X$ on $Y$, and let $m \in \Z$ be the multiplicity of $-K_X$ at $p$.
Then
\[
-K_Y = -\tilde K_X + (m-1)E.
\]
\end{lemma}

\begin{proof}
Follows from the ramification formula
\(
K_Y = \eta^* K_X + E.
\)
\end{proof}

\begin{proposition}
\label{p:base_points_on_poles}
Let $\phi\colon X_1 \rat X_2$ be a birational map between two smooth projective surfaces.
Let $\omega_1$ be a polar form on $X_1$.
Then $\phi_*(\omega_1)$ is a polar form on $X_2$ if and only if all base points of $\phi$ lie on pole of $\omega_1$.
\end{proposition}

\begin{proof}
If each base point lies on a pole of $\omega_1$, then for each blowup we can apply \cref{l:mult_blowup} with $m \ge 1$, and so each exceptional divisor appears with multiplicity $m-1 \ge 0$.

Conversely, assume that one of the base point $p_i$ does not lie on a pole of $\omega_1$. Then the corresponding exceptional divisor $E_i$ appears with multiplicity $-1$ in the anticanonical divisor, and all base points infinitely near to $p_i$ also admit a negative coefficient.
Let $p_j$ be a base point infinitely near to $p_i$ and without any infinitely near point above $p_j$ (possibly $p_j = p_i$).
Then the associated exceptional divisor $E_j$ projects to a curve in $X_2$, hence is part of the support of $D_2$, the anticanonical curve associated to $\omega_2=\phi_*(\omega_1)$, with negative coefficient.
\end{proof}

Let $(X, D)$ and $(X', D')$ be two Calabi--Yau pairs, with associated forms $\omega$, $\omega'$.
We say that a birational map $\phi\colon X \rat X'$ is \emph{volume preserving} if $\phi^* \omega' = \lambda \omega$ for some constant $\lambda \in \k^*$.
If $\phi\colon (X_1, \omega_1) \rat (X_2, \omega_2)$ is a volume preserving birational map, we say that $\omega_1$ and $\omega_2$ are \emph{birationally equivalent}.

\begin{example}
\label{e:quadratic}
Let $C_1\subset \P^2$ be a cubic curve, and $p_1$, $p_2$, $p_3$ be three proper noncollinear points on $C_1$.
Then by \cref{p:base_points_on_poles}, any quadratic map $\sigma\colon \P^2 \rat \P^2$ with these points as base points is a volume preserving map.
To determine the form $\omega_2$ on the target, one has to determine the new locus of poles, which is a cubic $C_2 \subset \P^2$.
By \cref{l:mult_blowup}, if the multiplicity of $C_1$ at $p_i$ is $m_i$, then the exceptional divisor carries a pole of order $m_i - 1$.
Moreover, the three contracted curves are the lines through two of the $p_i$.
For instance, in \cref{f:basic_quadratic_maps}, top left, we start from a cuspidal cubic $C_1$. After blowing up the singular point we get another curve of pole (tangent to the strict transform of the cubic), and the blowup of two smooth points does not add any pole.
After contraction, we get a cubic curve with two tangent components, which must be a conic and a tangent line.

Similarly, if $p_2$ is infinitely near to $p_1$, then either the curve $C_1$ has multiplicity $m_1 \ge 2$ at $p_1$ (\cref{f:basic_quadratic_maps} shows two such examples), or it is smooth at $p_1$ with tangent line in the direction of $p_2$ (see \cref{f:equal_type_4} for an example).
In diagrams, we represent an infinitely near point by an arrow corresponding to the tangent direction determined by the point.
\end{example}

\begin{lemma}
\label{l:jonquieres}
Let $D, D'$ be two anticanonical curves on $\P^2$, and let
$\phi\colon (\P^2, D) \rat (\P^2, D')$ be a volume preserving map.
Then there exists a factorisation $\phi = \phi_n \circ \dots \circ \phi_1$, where each $\phi_i\colon (\P^2,D_i) \rat (\P^2,D_{i+1})$ is a volume preserving Jonquières map, for some anticanonical curves $D_1 = D, \dots, D_{n+1} = D'$.
\end{lemma}

\begin{proof}
Follows from \cref{p:simplifying_jonquieres,p:base_points_on_poles}.
\end{proof}

\subsection{Birational classification}

The following is essentially \cite[Proposition B]{DillerLin}.
We take homogeneous coordinates $\pcoor{x:y:z}$ on $\P^2$ and use the classical embedding $\A^2\hookrightarrow \P^2, (x,y)\mapsto \pcoor{x:y:1}$.

\begin{proposition}
\label{p:birational_classification}
Let $\omega$ be a polar form on $\P^2$.
Then $\omega$ is birationally equivalent to one and only one of the following forms on $\P^2$:
\begin{description}
\item[Smooth case] A form $\omega_0$ with $-\div(\omega_0)$ a smooth cubic curve.
\item[Euclidean case] The form with affine expression $dx \wedge dy$ on $\A^2$, corresponding to the projective anticanonical curve $z^3 = 0$.
\item[Toric case] The form with affine expression $\frac{dx}{x} \wedge \frac{dy}{y}$ on $(\k^*)^2$, corresponding to the projective anticanonical curve $xyz = 0$.
\end{description}
Moreover:
\begin{itemize}
\item The form $\omega$ is birationally equivalent to the Euclidean case if and only if the anticanonical curve associated with $\omega$ is
\begin{enumerate}
\item a triple line,
\item or the union of a double line and another line,
\item or the union of three lines passing through a same point,
\item or the union of a smooth conic and a tangent line,
\item or a cuspidal cubic.
\end{enumerate}
\item The form $\omega$ is birationally equivalent to the toric case if and only if the anticanonical curve associated with $\omega$ is
\begin{enumerate}[resume]
\item the union of three lines not passing through a same point,
\item or the union of a smooth conic and a nontangent line,
\item or a nodal cubic.
\end{enumerate}
\end{itemize}

\end{proposition}

\begin{proof}
If $D = -\div(\omega)$ is a smooth cubic, $D$ is not contracted by any birational map $\phi \in \Bir(\P^2)$.
So if $\phi\colon (\P^2, D) \rat (\P^2, D')$ is a volume preserving map, $\phi_* \omega$ admits $\phi(D)$ as a pole, and $D' = \phi(D)$ must be a smooth cubic curve.

If $D = -\div(\omega)$ is a singular cubic, then there are eight possibilities up to projective equivalence, as listed in the statement.
The fact that inside the two groups the curves are birationally equivalent follows by explicitly checking that we can pass from one to the other by a sequence of volume preserving quadratic maps, see \cref{e:quadratic} and \cref{f:basic_quadratic_maps}.

Finally, assume by contradiction that there is a volume preserving birational map $\phi$ and a resolution diagram
\[
\begin{tikzcd}
& (Z,D) \ar[dl,"\eta_1",swap] \ar[dr,"\eta_2"]\\
(\P^2, C_1) \ar[rr,"\phi",dashed] && (\P^2, C_2)
\end{tikzcd}
\]
 where $C_1 = -\div(\omega_1)$ is a nodal cubic, $C_2=-\div(\omega_2)$ is a cuspidal cubic, and $(Z, D)$ is the minimal resolution of $\phi$.
The transform of the form $\omega_1$ on $Z$ must have a support with a cycle of rational curves, whereas the transform of the form $\omega_2$ does not carry such a cycle. This contradicts $\div(\eta_1^*\omega_1) = \div(\eta_2^*\omega_2) = -D$.
\end{proof}

\begin{figure}[ht]
\centering
\begin{tikzcd}
\includecurve{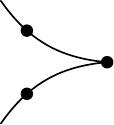} \ar[r, "\sigma"] &
\includecurve{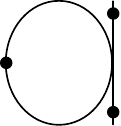} &
\includecurve{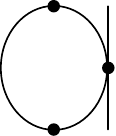} \ar[r, "\sigma"] &
\includecurve{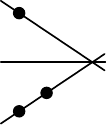}
\\
\includecurve{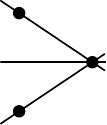} \ar[r, "\sigma"] &
\includecurve{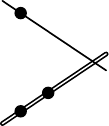} &
\includecurve{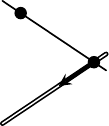} \ar[r, "\sigma"] &
\includecurve{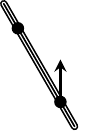}
\\
\includecurve{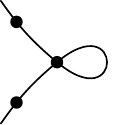} \ar[r, "\sigma"] &
\includecurve{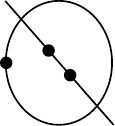} &
\includecurve{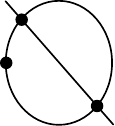} \ar[r, "\sigma"] &
\includecurve{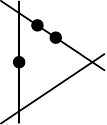}
\end{tikzcd}
\caption{Quadratic equivalences between forms on $\P^2$.}
\label{f:basic_quadratic_maps}
\end{figure}

\subsection{Links}\label{SubSec:Links}

The Hirzebruch surface of index $n \ge 0$ is the $\p^1$-bundle $\F_n \to \P^1$, which has a section of minimal self-intersection $-n$.
This property characterises $\F_n$ up to isomorphism, and any rational surface with a structure of $\P^1$-bundle is isomorphic to $\F_n$ for some $n\ge 0$.
The surface $\F_0$ is isomorphic to $\P^1 \times \P^1$, and $\F_1$ to the blowup of $\P^2$ at one point.

Given a point $p$ on $\F_n\to \p^1$, we call link (also called Sarkisov link of type II, or elementary transformation) centred at $p$ a birational map $\chi\colon\F_n \rat \F_{n\pm1}$ that fits in a commutative diagram
\[
\begin{tikzcd}
& Z \ar[dl] \ar[dr] \\
\F_n \ar[dr] \ar[rr,dashed,"\chi"] && \F_{n\pm1}, \ar[dl] \\
& \P^1
\end{tikzcd}
\]
where $Z \to \F_n$ is the blowup of $p$, and $Z \to \F_{n\pm1}$ is the contraction of the strict transform of the fibre through $p$.
When $n = 0$, the resulting surface is always $\F_1$, and when $n \ge 1$, the resulting surface is $\F_{n+1}$ (resp. $\F_{n-1}$) precisely if $p \in E$ (resp. $p \notin E$), where $E \subset \F_n$ is the section with $E^2 = -n$.
Observe that the structure of $\P^1$-bundle on the resulting surface $\F_{n\pm1}$ is uniquely determined by the diagram.

\begin{lemma}
\label{l:links_jonquieres}
Let $\phi \in \Bir(\P^2)$ be a Jonquières map of degree $d \ge 2$, and let $p,q \in \P^2$ such that $\phi$ sends the pencil of lines through $p$ to the pencil of line through $q$. Let $\eta\colon X_1=\F_1 \to \P^2$ be the blowup of $p$ and $\pi_1\colon X_1\to \p^1$ a $\p^1$-bundle $($unique up to automorphism of $\p^1)$.
Then there is a sequence of $2d -2$ links $\chi_1, \dots, \chi_{2d-2}$
\[
\chi_i\colon X_i\dasharrow X_{i+1}
\]
and $\p^1$-bundles $\pi_i\colon X_i\to \p^1$ such that $\pi_{i+1}\chi_i=\pi_i$ for each $i\in \{1,\dots,2d-2\}$ and such that
\[\phi= \sigma \chi_{2d-2} \dots \chi_1 \eta^{-1},\]
where $\sigma\colon X_{2d-1}\simeq \F_1 \to \P^2$ is a birational morphism contracting the negative section of $\pi_{2d-1}$ onto the point $q$.
Moreover, this factorisation is dominated by the minimal resolution of $\phi$.
\end{lemma}

\begin{proof}
One can invoke the general Sarkisov program, see e.g.~\cite[16.79]{CremonaBook}, but it is also easy to give a direct argument.
The linear system of $\phi$ is the system of plane curves of degree $d$ passing with multiplicity $d-1$ through $p$, and passing simply through $2d-2$ other base points.
After the blowup of $p$, this system becomes a system of sections on $\F_1$, with $2d-2$ simple base points.
Pick any proper base point $r$, and perform the link centred at $r$. This decreases the number of base points by one. After $2d-2$ such links, we get a birational morphism from a Hirzebruch surface to $\P^2$, which must be the blowup of $q$, so $X_{2d-1}\simeq \F_1$.
\end{proof}

Let $\F_n$ be the Hirzebruch surface of index $n$, and $\omega$ a $2$-form on $\F_n$.
We say that $\omega$ is a \emph{standard form} if $-\div(\omega) = 2 E + \sum_i m_i F_i$, for $E$ a section (not necessarily the exceptional section), $F_1, \dots, F_r$ some fibres, and $m_1, \dots, m_r \in \N$ some integers.
We write $\ord_\omega(F_i) = -m_i$.

\begin{lemma}
\label{l:one_link_on_form}
Let $\omega$ be a standard form on $\F_n$, $p \in \F_n$ a point, and $F$ the fibre containing $p$.
Let $\chi\colon \F_n \rat \F_{n\pm1}$ be the link centred at $p$, and write $F'$ the strict transform of the exceptional divisor above $p$.
Let $\omega' = \chi_*(\omega)$.
\begin{enumerate}
\item If $p \in E \cap F$, then $\ord_{\omega'}(F') = \ord_\omega(F)-1$.
\item If $p \in F \setminus E$, then $\ord_{\omega'}(F') = \ord_\omega(F)+1$.
\end{enumerate}
In particular, $\omega'$ is a standard form if and only if $p$ is on a pole of $\omega$.
\end{lemma}

\begin{proof}
This is a direct consequence of \cref{l:mult_blowup}.
\end{proof}

\begin{proposition}
\label{p:unimodal}
Let $\phi \in \Jonq(3L)$ be a volume preserving Jonquières map of degree $d \ge 2$, with a maximal base point $p\in L$, and let $\phi = \sigma \chi_{2d-2} \dots \chi_1 \eta^{-1}$ be a factorisation as in \cref{l:links_jonquieres}.
Let $E=\eta^{-1}(p)\subset X_1=\F_1$ be the exceptional divisor above $p$.
Let $\omega$ be a form on $\P^2$ with triple pole along $L$.
Then:
\begin{enumerate}
\item \label{p:unimodal:1}
On each Hirzebruch surface of the factorisation, the transform of $\omega$ is a standard form, with respect to the section given by the strict transform of $E$.
\item \label{p:unimodal:2}
The first and last Hirzebruch surfaces are $\F_1$ with standard form with poles of the form $-2E - 3 \tilde{L}$, where $\tilde{L}$ is the strict transform of $L$ by $\eta$ and $\sigma$ respectively.
\item \label{p:unimodal:3}
Given $u \in \P^1$, and $1 \le i \le 2d-1$, let $\omega_i$ be the transform on $X_i$ of the form $\omega$, $F_i=\pi_i^{-1}(u)$ the fibre over $u$ on $\F_{n_i}$. Then the sequence
\[-\ord_{\omega_1}(F_1),\dots, -\ord_{\omega_{2d-1}}(F_{2d-1})\] is unimodal \parent{increasing up to a point and then decreasing}.
\end{enumerate}
\end{proposition}

\begin{proof}
Since the factorisation of $\phi$ into links is dominated by the minimal resolution of $\phi$, by \cref{p:base_points_on_poles} each point blown up by a link is on the pole of the corresponding form.
Since the first and last links correspond to a blowup $\F_1 \to \P^2$ of a point on the triple line of pole, by \cref{l:mult_blowup} we get that the form on each $\F_1$ has pole of the form $-2E -3F$, and so is a standard form.
By \cref{l:one_link_on_form}, all forms on intermediate Hirzebruch surfaces are also standard forms. This proves \ref{p:unimodal:1} and \ref{p:unimodal:2}.

Now fix $u \in \P^1$, $1 \le i \le 2d-1$, and let $p_i$ be the centre of the link  $X_i \rat X_{i+1}$.
If $p_i \notin F_i = \pi_i^{-1}(u)$, then $\ord_{\omega_1}(F_{i+1}) = \ord_{\omega_1}(F_i)$.
If $p_i \in F_i \cap E$, then by \cref{l:one_link_on_form} $-\ord_{\omega_1}(F_{i+1}) = -\ord_{\omega_1}(F_{i}) + 1$.
Finally, if $p_i \in F_i \setminus E$, then by \cref{l:one_link_on_form} $-\ord_{\omega_1}(F_{i+1}) = -\ord_{\omega_1}(F_{i}) - 1$, and all subsequent blowups on this fiber are outside the section $E$ (otherwise we would have two links inverse of each other in the factorisation).
This proves \ref{p:unimodal:3}.
\end{proof}

\section{Generation in the Euclidean case}
\label{sec:generation}
In this section, we prove the generating part of \cref{t:main,t:main_bis}.

\begin{definition}
We denote by $\Jonq(3L)$ the set of Jonquières maps preserving the form $\omega$ with triple pole along $L$.
In other words, elements of $\Jonq(3L)$ are maps of the form $a b a'$, with $a, a' \in A$ and $b \in B$.
\end{definition}

\begin{theorem}
\label{t:generators}
The group $\Bir(\P^2, 3L)$ of volume preserving maps is generated by the affine group $A = \k^2 \rtimes \GL_2(\k)$ and the Jonquières group $B$  of \cref{t:main}, or equivalently by the set $\Jonq(3L)$ of volume preserving Jonquières maps.
\end{theorem}

\begin{figure}[ht]
\begin{tikzcd}
\includecurve{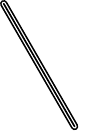}&
\includecurve{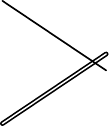}&
\includecurve{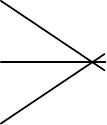}&
\includecurve{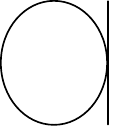}&
\includecurve{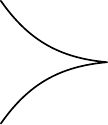}\\
(1)&(2)&(3)&(4)&(5)
\end{tikzcd}
\caption{Five possible anticanonical curves.}
\label{f:5_curves}
\end{figure}

By \cref{p:birational_classification}, there are five possibilities for an anticanonical curve on $\P^2$ birationally equivalent to a triple line.
Given such an anticanonical curve $D$, we denote by $T(D) \in \{1,2,3,4,5\}$ the corresponding case, as indicated on \cref{f:5_curves}.
We say that $T(D)$ is the \emph{type} of $D$.
If $f \colon (\P^2, D) \rat (\P^2, D')$ is a volume preserving map with $t= T(D)$, $t' = T(D')$, we say that $f$ has type $(t,t')$.
We put an order on types by taking the lexicographic order on the pair $( \max\{t,t'\},\min\{t,t'\})$.
For instance $(5,3) > (4,4) > (3,4)$, and $(3,4)$ and $(4,3)$ are equivalent for this order.

\begin{lemma}
\label{l:stays_same}
Let $D$ be an anticanonical curve of type $t \in \{2, 4\}$ and let $p, r, r'$ be points of $D$, proper or infinitely near, chosen as follows:
\begin{enumerate}
\item \label{l:stays_same_4}
If $t = 4$, writing $D = C \cup L$ with $C$ a conic and $L$ a tangent line, we take $p$ to be a point on $L \setminus C$, $r$ to be a point on $C \setminus L$, and $r'$ to be either another point on $C$, or the unique point on $C$ in the first neighbourhood of $r$.
See \cref{f:equal_type_4}.
\item \label{l:stays_same_2}
If $t = 2$, writing $D = L_1 + 2L_2$, we take $p$ to be a point on $L_2 \setminus L_1$, $r$ to be a point on $L_1$, and $r'$ either another point on $L_1$ or the unique point on $L_1$ in the  first neighbourhood of $r$.
See \cref{f:equal_type_2}.
\end{enumerate}
Then, there  exists a volume preserving quadratic map $\tau\colon (\P^2,D)\dasharrow (\P^2,D)$ with base points $p,r,r'$ that sends the pencil of lines through $r$ either onto itself or to the pencil of lines through the singular point of $D$. 
\end{lemma}

\begin{figure}[ht]
\centering
\begin{tikzcd}[row sep = 5]
\includecurve{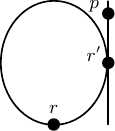} &
\includecurve{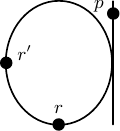} &
\includecurve{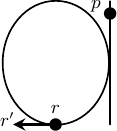}
\end{tikzcd}
\caption{Three cases of \cref{l:stays_same}\ref{l:stays_same_4}.}
\label{f:equal_type_4}
\end{figure}

\begin{figure}[ht]
\centering
\begin{tikzcd}[row sep = 5]
\includecurve{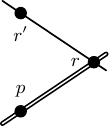} &
\includecurve{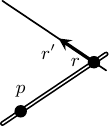} &
\includecurve{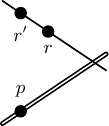} &
\includecurve{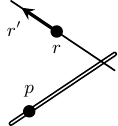}
\end{tikzcd}
\caption{Four cases of \cref{l:stays_same}\ref{l:stays_same_2}.}
\label{f:equal_type_2}
\end{figure}

\begin{proof}
Denote by $\omega$ the form associated to $D$.
We consider a quadratic map $\tau\in \Bir(\P^2)$ whose base points are $p,r,r'$. As all base points of $\tau$ are poles of $\omega$, by \cref{p:base_points_on_poles} we obtain a volume preserving birational map $\tau\colon (\P^2,D)\dasharrow (\P^2,D')$.

We prove first that $T(D')=T(D)$.
In Case~\ref{l:stays_same_4}, this is because the conic $C$ is sent to a conic by $\tau$.
In Case~\ref{l:stays_same_2}, this is because $D'$ contains a double line, either the image of the double line of $D$ (last two cases of Figure~\ref{f:equal_type_2}), or corresponding to the exceptional divisor of $r$ (first case of Figure~\ref{f:equal_type_2}) or $r'$ (second case of Figure~\ref{f:equal_type_2}). 

Since $T(D')=T(D)$, we may replace $\tau$ with $\alpha\tau$ for some $\alpha\in \Aut(\P^2)$ and assume that $D'=D$.
We denote by $\Lambda$ the pencil of lines through $r$. As $r$ is a base point of $\tau$, the pencil $\tau(\Lambda)$ is again a pencil of lines through $\tilde{r}\in \P^2$, that is a base point of $\tau^{-1}$. Applying  \cref{p:base_points_on_poles} to $\tau^{-1}$, the point $\tilde{r}$ belongs to $D$.

(i) If $\tilde{r}$ is the singular point of $D$, we are done.

(ii)  If  $\tilde{r}$ and $r$  belong both to $C\setminus L$ (respectively $L_1\setminus L_2$),  we replace $\tau$ with $\beta\tau$ where $\beta\in \Aut(\P^2)$ preserves $D$ and sends $\tilde{r}$ to $r$, and we are done.

It remains to see either (i) or (ii) always occurs.

If $t=4$, we only need to show that the case $\tilde{r}\in L\setminus C$ leads to a contradiction.  This follows from the fact that in this case the rational map $C\dasharrow \p^1$ induced by the pencils $\Lambda$ and  $\tau(\Lambda)$ is birational in the first case and not in the second.

If $t=2$, we only need to show that the case $\tilde{r}\in L_2\setminus L_1$ leads to a contradiction.
In that case, $L_2$ belongs to the pencil $\tau(\Lambda)$.
In the last two cases of Figure~\ref{f:equal_type_2}, $L_2$ is invariant by $\tau$ but corresponds to a section of $\Lambda$.
In the first case of Figure~\ref{f:equal_type_2}, the preimage of $L_2$ corresponds to the exceptional divisor of $r$, but this one is again a section.
It remains to see that in the second case of Figure~\ref{f:equal_type_2}, the pencil of lines through $r$ is invariant, which follows from the fact that the form has a pole of order $3$ on the exceptional divisor of~$r$.
\end{proof}

\begin{definition}
\label{d:pointed_quad_map}
For each $t \in \{2,3,4,5\}$, we choose a family of quadratic maps $\sigma_{t,p}$ of type $(t,s)$ where $t > s$, and where the point $p$ used as index is one of the three base points, together with $p_0$ the singular point, and $p_1$ is the first neighbourhood of $p_0$.
Precisely:
\begin{itemize}
\item If $t = 5$, then $p_1$ corresponds to the direction of the cusp, and $p$ is any smooth point.
The corresponding $\sigma_{5,p}$ has type $(5,4)$.
\item If $t = 4$, then $p_1$ corresponds to the direction of the line, and $p$ is any smooth point on the conic.
The corresponding $\sigma_{4,p}$ has type $(4,2)$.
\item If $t = 3$, then $p_1$ corresponds to the direction of one of the three lines, and $p$ is any smooth point on the remaining two lines.
The corresponding $\sigma_{3,p}$ has type $(3,2)$.
\item If $t = 2$, then $p_1$ corresponds to the direction of the double line, and $p$ is any smooth point on the other line.
The corresponding $\sigma_{2,p}$ has type $(2,1)$.
\end{itemize}
We call any such map $\sigma_{t,p}$ a \emph{pointed quadratic map}.
See \cref{f:sigma_t}.
\end{definition}

\begin{figure}[t]
\centering
\begin{tikzcd}
\includecurve{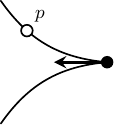} \ar[r,"\sigma_{5,p}"] &
\includecurve{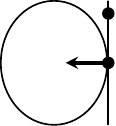} &
\includecurve{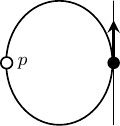} \ar[r,"\sigma_{4,p}"]&
\includecurve{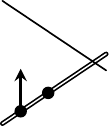}
\\
\includecurve{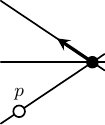} \ar[r,"\sigma_{3,p}"]&
\includecurve{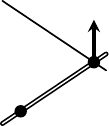} &
\includecurve{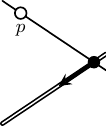} \ar[r,"\sigma_{2,p}"]&
\includecurve{1_triple_d_1}
\end{tikzcd}
\caption{Pointed quadratic maps $\sigma_{t,p}$, with mobile points $p$ indicated by $\boldsymbol{\circ}$.}
\label{f:sigma_t}
\end{figure}

\begin{remark}
Given a configuration $C + L$ of type $4$ and a point $p \in L \setminus C$, there is no quadratic map with base point at $p$ that is volume preserving of type $(4,m)$, with $m < 4$.
Indeed, the image of the conic $C$ would be a curve of degree $\ge 2$ if not all base points are on $C$.

Similarly, given a configuration $L_1 + 2L_2$ of type $2$ and a point $p \in L_2 \setminus L_1$, there is no quadratic map with base point at $p$ that is volume preserving of type $(2,1)$.
This can be proven by a case-by-case analysis.

This explains why the proof of \cref{p:master} below is a little more involved when $t \in \{2,4\}$.
\end{remark}

\begin{proposition}
\label{p:master}
Let $f$ be a volume preserving map of the form $\phi \circ \sigma_{t,p}^{-1}$, where $\sigma_{t,p}$ is a pointed quadratic map of type $(t,s)$ with $t > s$, and $\phi $ is a Jonquières map of type $(t, u)$, with $t \ge u$.
Then $f$ is a composition of Jonquières maps of types $(s,s)$ and $(s,u)$.
\end{proposition}

\begin{proof}
The singular point $p_0$ of the type $t$ anticanonical curve is a base point of $\sigma_{t,p}$. Hence, if $p_0$ is a maximal base point of $\phi$, then $\phi \circ \sigma_{t,p}^{-1}$ is a Jonquières map of type $(s,u)$, and we are done.
Indeed, denoting by $\Pl$ the pencil of lines through $p_0$, by assumption both $\sigma_{t,p}(\Pl)$ and $\phi(\Pl)$ are also pencil of lines, and $\phi \circ \sigma_{t,p}^{-1}$ is Jonquières because its sends $\sigma_{t,p}(\Pl)$ onto $\phi(\Pl)$.
In the sequel we prove similarly that various compositions of maps are Jonquières, without detailing the involved pencils.

If $\phi$ has a maximal point $q$ that corresponds to the mobile point of some pointed quadratic map $\sigma_{t,q}$, then we write
\[
\phi \circ \sigma_{t,p}^{-1}
= (\phi \circ \sigma_{t,q}^{-1}) \circ (\sigma_{t,q} \circ \sigma_{t,p}^{-1}).
\]
This is a composition of two Jonquières maps of respective types $(s,u)$ and $(s,s)$, so we are done also in this case.
Observe that if $t \in \{3, 5 \}$, then we already covered all possible cases, and the proof is finished.
Similarly, if $\deg(\phi ) = 2$, then any proper base point of $\phi $ is a maximal point, and again we can apply the above argument.
So we can assume $\deg(\phi ) \ge 3$, and that the proposition is proven for any Jonquières map of degree strictly less than $\phi $.
Denoting by $p$ the unique maximal base point of $\phi $, we are left with the following two configurations:
\begin{itemize}
\item $t = 4$, and writing $C + L$ the configuration of type $4$, $p \in L \setminus C$
\item $t = 2$, and writing $L_1 + 2L_2$ the configuration of type $2$, $p \in L_2 \setminus L_1$.
\end{itemize}
In each case, we define two other base points $r,r'$ of $\phi$ as follows, and apply \cref{l:stays_same} to obtain a volume preserving quadratic map $\tau\colon (\P^2,D)\dasharrow (\P^2,D)$  with base points $p,r,r'$.
Moreover, $\tau$ sends the pencil of lines through $r$ either onto itself or to the pencil of lines through the singular point of $D$.

In the case $t = 4$, since $\phi(C)$ is either a point, a line or a conic, there are at least two base points $r,r'$ of $\phi$ on $C$.
Moreover, there are at most two base points of $\phi$ on  $L$, so we can assume that $r$ is a proper point on $C \setminus L$, and $r'$ either in the first neighbourhood of $r$ or is a proper point on $C$ (possibly $C \cap L$).
This corresponds to Case~\ref{l:stays_same_4} of \cref{l:stays_same}.

Similarly, in the case $t = 2$, since $\phi(L_1)$ is a line or a point, the line $L_1$ passes through at least two base points $r,r'$ of $\phi$.
Moreover, we can choose $r$ to be a proper point, and $r'$ to be proper or in the first neighbourhood of $r$.
This corresponds to Case~\ref{l:stays_same_2} of \cref{l:stays_same}.

Then in both cases $\phi' = \phi \circ \tau^{-1}$ is a Jonquières map of type $(t,u)$, with $\deg(\phi') = \deg(\phi) - 1$ (\cref{PropDelta}).

If $r$ is the singular point of the anticanonical curve, then $\tau$ preserves the pencil of lines through $r$. Hence, $\sigma_{t,p} \circ \tau \circ \sigma_{t,p}^{-1}$ is a Jonquières map of type $(s,s)$, and writing
\[
\phi \circ \sigma_{t,p}^{-1} = \phi' \circ \tau \circ \sigma_{t,p}^{-1} = (\phi' \circ \sigma_{t,p}^{-1}) \circ (\sigma_{t,p} \circ \tau \circ \sigma_{t,p}^{-1}),
\]
we conclude by the induction hypothesis.
Otherwise, $r$ is general point of the conic ($t=4$) or the simple line ($t=2$), hence can be taken as the mobile point of a pointed quadratic map $\sigma_{t,r}$.
As $\tau$ sends the pencil of lines through $r$ either onto itself or to the pencil of lines through the singular point of $D$, and as both $r$ and the singular point of $D$ are base points of the quadratic map $\sigma_{t,r}$, the map $\sigma_{t,r} \circ \tau \circ \sigma_{t,r}^{-1}$ is Jonqui\`eres of type $(s,s)$.
We write
\[
\phi \circ \sigma_{t,p}^{-1} = \phi' \circ \tau \circ \sigma_{t,p}^{-1} = (\phi' \circ \sigma_{t,r}^{-1}) \circ (\sigma_{t,r} \circ \tau \circ \sigma_{t,r}^{-1}) \circ (\sigma_{t,r} \circ \sigma_{t,q}^{-1}).
\]
The last two factors are Jonquières of type $(s,s)$, and again we conclude by the induction hypothesis.
\end{proof}

\begin{proof}[Proof of \cref{t:generators}]
Let $f \in \Bir(\P^2,3L)$.
By \cref{l:jonquieres}, there exists a factorisation $f = \phi_n \circ \dots \circ \phi_1$ where each $\phi_i\colon (\P^2,D_i) \rat (\P^2,D_{i+1})$ is a volume preserving Jonquières map, for some anticanonical curves $D_1 = 3L, D_2 \dots, D_n, D_{n+1} = 3L$.
By \cref{p:birational_classification}, each $D_i$ is one of the five curves on \cref{f:5_curves}.
We denote by $t_i \in \{1,2,3,4,5\}$ the type of the curve $D_i$, so that each $\phi_i$ in the factorisation is a Jonquières map of type $(t_i, t_{i+1})$.

If all $\phi_i$ have type $(1,1)$, we are done.
Otherwise, consider $t \ge u$ realising the maximal type $(t,u)$ or $(u,t)$ among all $\phi_i$.
For each $D_i$ of type $t$, we insert $\sigma_i^{-1} \circ \sigma_i$ in the factorisation, with $\sigma_i = \sigma_{t,p_i}$ a pointed quadratic map from $(\P^2, D_{i})$.
Then by applying \cref{p:master} to each such $\phi_i \circ \sigma_i^{-1}$ (or to the inverse of $\sigma_i \circ \phi_i$ when $\phi_i$ has type $(u,t)$), we obtain a new factorisation with strictly lower maximal type.
We conclude by induction.
\end{proof}

\section{Jonquières maps of \texorpdfstring{$\Bir(\P^2,3L)$}{Bir(P2,3L)}}
\label{s:volume preserving-jonquieres-maps}

In this section we discuss the possible configurations of base points for elements of the set $\Jonq(3L)$, which was defined in Section \ref{sec:reformulation}.
Recall that we fix a choice of line $L=\{z=0\} \subset \P^2$, and fix an embedding $\mathbb{A}^2 \hookrightarrow \P^2, (x,y)\mapsto \pcoor{x:y:1}$ that gives an isomorphism $\A^2 \iso \P^2 \setminus L$.

\subsection{The group \textit{B} in affine coordinates}
We give the explicit form in affine coordinates for the group $B$ of \cref{t:main}:

\begin{lemma}
\label{l:jonq_preserving}
The subgroup $B \subset \Bir(\P^2,3L)$ of volume preserving Jonquières maps preserving the pencil of lines through $\pcoor{1:0:0}$ consists of maps of the form
\[
\left( \frac{\lambda}{M'(y)}x + F(y), M(y)\right),
\]
where $F(y) \in \k(y)$, $\lambda \in \k^*$, and $M$ is a Möbius map that we identify with an element in $\PGL_2(\k)$.
\end{lemma}

\begin{proof}
For any Möbius map $M(y)=\frac{a y + b}{c y + d}$ with $\smallmat{a & b\\c & d} \in \PGL_2(\k)$, one checks that
\[
\phi = \left( \frac{x}{M'(y)} , M(y)\right)\in \Bir(\P^2)
\]
satisfies $\phi^* \omega = \omega$, where $\omega = dx \wedge dy$.
So it is sufficient to characterise volume preserving Jonquières maps of the form
\[
\psi = \left( \frac{\alpha x + \beta}{\gamma x + \delta}, y \right) \text{ with } \smallmat{\alpha & \beta\\\gamma & \delta} \in \PGL_2(\k(y)).
\]
We compute
\[
\psi^* \omega= \frac{\alpha \delta - \beta \gamma}{(\gamma x +\delta)^2}\, \omega.
\]
Hence, $\psi^* \omega =\lambda. \omega$, with $\lambda\in\k^*$, if and only if $\frac{\alpha \delta - \beta \gamma}{(\gamma x +\delta)^2}=\lambda$.
This is only possible if $\gamma=0$.
In that case, we can choose $\delta=1$, and we obtain the expected form with $\alpha=\lambda$, $\beta(y) = F(y)$.
\end{proof}

\begin{remark}
Explicitly, if $M(y) = \frac{ay +b}{cy+d}$, then $M'(y) = \frac{ad - bc}{(cy+d)^2}$.
\end{remark}

\begin{proposition}\label{prop:abelianisation-trivial}
Recall that $\omega=dx\wedge dy$ satisfies $\div(\omega)=-3L$. The abelianisation morphism of $\Bir(\P^2, 3L)$ is the Jacobian map
 \[\mathrm{Jac}\colon \Bir(\P^2, 3L)\to \k^*,\]
 that sends $f\in\Bir(\P^2, 3L)$ onto $\lambda\in \k^*$, where $f^*(\omega)=\lambda \omega$.
\end{proposition}

\begin{proof}
By \cref{t:generators}, $\Bir(\P^2, 3L)$ is generated by the subgroup $A$ of affine maps and by the subgroup $B$ of Jonquières maps preserving the pencil through $[0:1:0]$.
By \cref{l:jonq_preserving}, $B$ consists of maps of the form
 \[
 (x,y) \mapsto \left( \frac{\lambda x}{M' (y)} + F(y), M(y) \right),
 \]
 with $M(y) \in \PGL_2 (\k)$ a Möbius transformation, and $F(y) \in \k(y)$ a rational function.
 
 We need to prove that every element of $A$ or $B$ of Jacobian $1$ is in the commutator subgroup $N\subseteq \Bir(\P^2, 3L)$.
 
 First, the commutator of $f=(x+by,y)$ and $g=(x,y+1)$ is $[f,g]=f^{-1}g^{-1}fg=(x+b,y)\in N$, for each $b\in \k$.
 We then consider $\mu\in \k^*$, $i\in \Z\setminus \{0\}$ and $a\in \k^*$, with $a^i\not=1$. The commutator of $f=(x+\mu y^i,y)$ and $g=(x,ay)$ is $[f,g]=f^{-1}g^{-1}fg=(x+\mu (a^i-1)y^i,y).$ 
 This proves that $(x+F(y),y)\in N$ for each $R\in \k[y,\frac{1}{y}]$. 
 
 In particular, we obtain $(x+cy+d,y)\in N$ for all $c,d\in \k$. Conjugating by $(y,x)$ we also obtain $(x,y+cx+d)\in N$. 
 These elements generate $\k^2 \rtimes \SL_2(\k)=\{\alpha\in A\mid \mathrm{Jac}(\alpha)=1\}$, so we obtain $\k^2 \rtimes \SL_2(\k)\subseteq N$.

 As $\PGL_2(\k)$ is a simple group, the subgroup $\{(\frac{x}{M'(y)}, M(y))\mid M\in \PGL_2(\k)\}$ of $B$ is contained in $N$. 
 
It remains to show that the group $H=\{(x+F(y),y)\mid F\in \k(y)\}\simeq (\k(y),+)$, is contained in $N$. We have already proven that $\k[y,\frac{1}{y}]\subseteq N\cap H\subseteq H$. Conjugating $H$ with $(x,y+c)$ sends $F\in \k(y)$ onto $F(y+c)$. Hence, $\k[y,\frac{1}{y+c}]\subseteq N\cap H$ for each $c\in \k$. This yields $N\cap H=\k(y)$ because of partial fraction decompositions for rational functions.
\end{proof}

\subsection{Base points of Jonquières maps}

\begin{proposition}
\label{Prop:BasePtsJonq}
Let $\phi\in \Jonq(3L)$ be a Jonquières map of degree $d\ge 2$, with maximal base point $p\in L$. Then, either $p$ is the unique proper base point, or $\phi$ has exactly two proper base points $p,t\in L$.
The $2d-1$ base points of $\phi$ are in one of the three configurations depicted on \cref{f:BasePtsJonq}, with the following conditions:
\begin{itemize}
\item $m\ge 1$ in Cases~\ref{BasePtsJonq1} and \ref{BasePtsJonq2}, and $m\ge 2$ in Case~\ref{BasePtsJonq3}.
\item $d_1,\dots,d_m\ge 1$, and $d_1\ge 4$ in Cases~\ref{BasePtsJonq2} and~\ref{BasePtsJonq3}, $d_2\ge 2$ in Case~\ref{BasePtsJonq3}.
\item $\sum_{i=1}^m (d_i-1)=d-1$.
\item For each $i = 1, \dots, m$, the points $p_{i,1},\dots,p_{i,d_i-1}$ are proximate to $p$, and for each $j \ge d_i$, the point $p_{i,j}$ is only proximate to $p_{i,j-1}$.
\item In Case \ref{BasePtsJonq2}, $t_2$ is not proximate to $t$.
\item In Case \ref{BasePtsJonq3}, the point $p_{2,1}$ belongs to the strict transform of $L$.
\item No three base points of $\phi$ are collinear, except maybe $t,t_1,t_2$ in Case~\ref{BasePtsJonq2}.
\end{itemize}

Conversely, for each configuration of points as above, there exists a Jonquières map $\varphi\in \Jonq(3L)$ of degree $d$ with these base points, where $p$ has multiplicity $d-1$, and all the other points have multiplicity $1$.
\end{proposition}

\begin{figure}[htb]
\centering
\begin{minipage}{10cm}
\begin{enumerate}[itemsep=3mm]
\item\label{BasePtsJonq1}
\begin{tikzcd}[row sep = 2]
 & p_{1,1} \ar[ld] & p_{1,2}\ar[l] & \ar[l] \cdots & \ar[l]p_{1,2d_1-2}\\
p& \vdots\\
 & p_{m,1}\ar[lu]&p_{m,2}\ar[l] &\ar[l] \cdots & \ar[l]p_{m,2d_m-2}
\end{tikzcd}
\hrule
\item\label{BasePtsJonq2}
\begin{tikzcd}[row sep = 2]
&p_{1,1}\ar[ldd] & p_{1,2}\ar[l]&\ar[l] \cdots & \ar[l]p_{1,2d_1-5}\\
 &p_{2,1}\ar[ld] & p_{2,2}\ar[l]&\ar[l] \cdots & \ar[l]p_{2,2d_2-2}\\
 p& \vdots\\
 &p_{m,1}\ar[lu]&p_{m,2}\ar[l] &\ar[l] \cdots & \ar[l]p_{m,2d_m-2}\\
t & t_1\ar[l] & t_2\ar[l]
\end{tikzcd}
\hrule
\item\label{BasePtsJonq3}
\begin{tikzcd}[row sep = 2]
& p_{1,1} \ar[ldd] & p_{1,2} \ar[l] & \ar[l] \cdots & p_{1, 2d_1 -5} \ar[l] \\
& p_{2,1} \ar[ld] & p_{2,2} \ar[l] & \cdots \ar[l] & p_{2, 2d_2 +1} \ar[l] \\
p & p_{3,1} \ar[l] & p_{3,2} \ar[l] & \cdots \ar[l] & p_{3, 2d_3 -2} \ar[l] \\
& \vdots \\
& p_{m,1} \ar[uul] & p_{m,2} \ar[l] & \cdots \ar[l] & p_{m, 2d_m -2} \ar[l]
\end{tikzcd}
\end{enumerate}
\end{minipage}
\caption{Configurations of base points.}
\label{f:BasePtsJonq}
\end{figure}
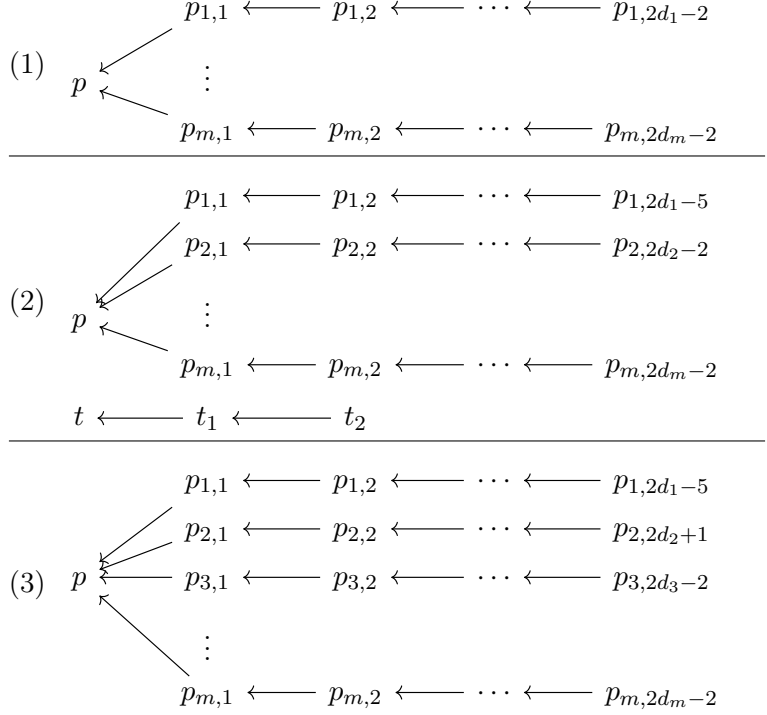

\begin{proof}
We decompose $\phi$ as $\phi =\sigma \chi_{2d-2} \dots \chi_1 \eta^{-1}$, as in \cref{l:links_jonquieres}, where $\eta$ is the blowup of the point $p\in L$, and $\sigma$ the blowup of a point $q \in L$ (\cref{p:base_points_on_poles}).
The exceptional divisors $\eta^{-1}(p)\subset X_{1}$ and $\sigma^{-1}(q)\subset X_{2d-1}$ are the only sections with poles.
Hence, $\chi_{2d-2} \dots \chi_1 $ sends the first section on the second.
We denote by $E$ the section with a pole in each $X_i$, which is always the strict transform of $\eta^{-1}(p)$.

We define $a,b\in \P^1$ as $a=\pi_1((\eta^{-1})_*{L})$ and $b=\pi_{2d-1}((\sigma^{-1})_*{L})$.
These are the points corresponding to the line at infinity at the source and the target.

For each $u\in \P^1$, the sequence $\mu_1, \dots, \mu_{2d-1}$ where $\mu_i = -\ord_{\omega_i}(\pi_i^{-1}(u))$ is unimodal by \cref{p:unimodal}.
If it is not constant, we say that the fibre above $u$ has type $(\mu_1, \mu_{2d-1})$.
Observe that $\mu_1, \mu_{2d-1} \in \{0,3\}$.
We have $\mu_1 = 3$ only if $u =a$, and $\mu_{2d-1} = 3$ only if $u=b$.

If the fibre above $u \in \P^1$ has type $(0,0)$ or $(3,3)$, half of the base points above $u$ lie on $E$, and the other half outside.
We can thus write these points as $p_{i,1}\leftarrow p_{i,2}\leftarrow \dots \leftarrow p_{i,2d_i-2}$, where for each $i$, the points $p_{i,1},\dots,p_{i,d_i-1}$ are proximate to $p$, but each other point $p_{i,j}$ is only proximate to the previous one $p_{i,j-1}$.
Moreover, $p_{i,2}$ is not collinear with $p_{i,1}$ and $p$, otherwise the link centred at $p_{i,2}$ would not exist.

If $a = b$, then the fibre above $a$ either does not contain any base point, or has type $(3,3)$.
Any other fibre containing a base point has type $(0,0)$.
So this case is completely covered by the above discussion, this is Case~\ref{BasePtsJonq1}.

Finally, we consider the case $a \neq b$.
Then the fibre above $a$ has type $(3,0)$, the fibre above $b$ has type $(0,3)$, and any other fibre has type $(0,0)$.
The sequence of $b$ increases from $0$ to $d_1-1\ge 3$, then decreases from $d_1-1$ to $3$.
This gives a sequence of points $p_{1,1}\leftarrow p_{1,2}\leftarrow \dots \leftarrow p_{1,2d_1-5}$, where $p_{1,1},\dots,p_{1,d_1-1}$ are proximate to $p$, but the other $p_{1,j}$ (if they exist) are only proximate to $p_{1,j-1}$.
We now have two subcases, according whether the sequence of the fibre $a$ is decreasing, or truly unimodal. Equivalently, this depends whether the first base point on the fibre above $a$ is outside $E$, or on $E$.

If the sequence of $a$ directly decreases, it consists of three points $t\leftarrow t_1 \leftarrow t_2$, where $t\in L\setminus \{p\}$, $t_1$ does not lie on the line $L$ through $t$ and $p$, and $t_2$ is not proximate to $t$.
This gives Case~\ref{BasePtsJonq2}.

If the sequence of $a$ starts by increasing from $3$ to $3+d_2-1$, with $d_2\ge 2$, it gives a sequence $p_{2,1}\leftarrow p_{2,2}\leftarrow \dots \leftarrow p_{2,2d_2+1}$, where $p_{2,1},\dots,p_{2,d_2-1}$ are proximate to $p$, but the other $p_{2,j}$ are only proximate to $p_{2,j-1}$. Moreover, $p_{2,1}$ belongs to the line $L$.
This gives Case~\ref{BasePtsJonq3}.

Conversely, choosing the points as above, we blow up $p$, then perform links at all the points. The conditions are chosen such that we can perform the links after each other, and ensure that after the links $\chi_1,\dots,\chi_{2d-2}$ we end-up with a $\P^1$-bundle $\pi_{2d-1}\colon X_{2d-1}\to \P^1$ with a section $E$ having a pole of order $2$, one fibre with a pole of order $3$ and all other fibres being neither a pole nor a zero. Moreover, half of the points being on $E$, we find $E^2=-1$ on $X_{2d-1}$.
Hence, $X_{2d-1}\simeq \F_1$, and there is a birational morphism $\sigma\colon X_{2d-1}\to \P^2$ sending the fibre with a triple pole onto $L$.
We obtain a Jonquières map $\phi=\sigma\chi_{2d-1}\dots \chi_1\eta^{-1}\in \Bir(\P^2,3L)$ with the desired base points.
\end{proof}

\begin{definition}
\label{d:tail}
In Case \ref{BasePtsJonq2} of \cref{Prop:BasePtsJonq}, we say that the three base points $t \leftarrow t_1 \leftarrow t_2$ form the \emph{tail} of the Jonquières map $\phi$.
\end{definition}

\begin{corollary}\label{Cor:SimplificationJonqDec}
Let $\Lambda$ be a homaloidal linear system on $\P^2$, and let $\phi\in \Jonq(3L)\setminus A$ be a Jonquières map with maximal base point $p\in L$, such that
\[\deg(\phi(\Lambda))\le \deg(\Lambda).\]
There exists a Jonquières map $\psi\in \Jonq(3L)\setminus A$ with maximal base point $p$, and with all of its base points among those of $\phi$, such that
\[
\deg(\psi(\Lambda))\le \deg(\Lambda),
\]
and such that the following hold:
\begin{enumerate}
\item\label{CorSimp1}
 The map $\psi$ corresponds to Case~\ref{BasePtsJonq1} or \ref{BasePtsJonq2} of \cref{Prop:BasePtsJonq} with $m=1$, or Case~\ref{BasePtsJonq3} of \cref{Prop:BasePtsJonq} with $m=2$.
 \item\label{CorSimp2}
   If $\deg(\phi(\Lambda))< \deg(\Lambda)$, then $\deg(\psi(\Lambda))< \deg(\Lambda)$.
\item\label{CorSimp3}
If $\psi$ has a tail $t\leftarrow t_1\leftarrow t_2$, then either $m_t(\Lambda)+m_{t_1}(\Lambda)+m_{t_2}(\Lambda)<\deg(\Lambda)$, or $\deg(\psi)=4$.
\end{enumerate}
\end{corollary}

\begin{proof}Write $D=\deg(\Lambda)$, $d=\deg(\phi)$. The linear system of $\phi$ consists of curves of degree $d$ with multiplicity $d-1$ at $p$ and multiplicity $1$ at $2d-2$ other points (\cref{l:noether_rel}).
Denoting by $\mathcal{B}$ the set of base points of $\varphi$ different from $p$ and writing $m_u=m_u(\Lambda)$ the multiplicity of $\Lambda$ at a point $u$, we find
\begin{align*}
D-\deg(\phi(\Lambda))&\stackrel{\text{Proposition~\ref{PropDelta}}}{=}\sum\nolimits_{b\in \mathcal{B}} m_b-(D-m_p)\cdot (d-1)=\sum\nolimits_{b\in \mathcal{B}} s(b),
\end{align*}
where $s(b)=m_b-\frac{D-m_p}{2}$. By assumption, we have $\sum_{b\in \mathcal{B}} s(b)\ge 0$.

We now apply \cref{Prop:BasePtsJonq} and decompose the set $\mathcal{B}$ into $m$ disjoint sets $\mathcal{B}_1,\dots,\mathcal{B}_m$ $(m-1$ sets $\mathcal{B}_1,\mathcal{B}_3,\dots,\mathcal{B}_m$ in Case~$\ref{BasePtsJonq3})$. We define $\mathcal{B}_i=\{p_{i,1},\dots,p_{i,2d_i-2}\}$ for each $i$, except in Case~\ref{BasePtsJonq2} for $i=1$, where $\mathcal{B}_1=\{p_{1,1},\dots,p_{i,2d_1-5},t,t_1,t_2\}$, and in Case~\ref{BasePtsJonq3}, where $\mathcal{B}_1=\{p_{1,1},\dots, p_{1, 2d_1 -5},p_{2,1},\dots, p_{2, 2d_2 +1} \}$, and where there is no $\mathcal{B}_2$.

As $\sum_{b\in \mathcal{B}} s(b)\ge 0$, there exists $i\in \{1,\dots,m\}$ $(i\not=2$ in Case~$\ref{BasePtsJonq3})$ such that $\sum_{b\in \mathcal{B}_i} s(b)\ge 0$ (with a strict inequality if the first inequality is strict).
\cref{Prop:BasePtsJonq} yields the existence of a Jonquières map $\psi\in \Bir(\P^2,3L)\setminus A$ with maximal base point $p$ having $p\cup \mathcal{B}_i$ as base points.
We then find $\deg(\phi(\Lambda))\le D$, with a strict inequality if $\deg(\phi(\Lambda))< D$.
This proves~\ref{CorSimp1} and \ref{CorSimp2}.

To prove~\ref{CorSimp3}, assume that $\psi$ has a tail $t\leftarrow t_1\leftarrow t_2$, that $k=\deg(\psi)>4$ and that $m_t+m_{t_1}+m_{t_2}\ge D$. The base points of $\psi$ are thus
\[t\leftarrow t_1\leftarrow t_2, \quad p\leftarrow p_{1,1}\leftarrow p_{1,2}\leftarrow \cdots \leftarrow p_{1,2k-5}.\]
\cref{Prop:BasePtsJonq} yields the existence of $\psi'\in \Jonq(3L)$ of degree $4$ with base points
 \[t\leftarrow t_1\leftarrow t_2, \quad p\leftarrow p_{1,1}\leftarrow p_{1,2}\leftarrow p_{1,3}.\]
 We replace $\psi$ with $\psi'$ and need to show that $\deg(\psi(\Lambda))\le D\Rightarrow \deg(\psi'(\Lambda))\le D$ and $\deg(\psi(\Lambda))< D\Rightarrow \deg(\psi'(\Lambda))< D$.
Writing $\mathcal{B}=\{t,t_1,t_2,p_{1,1},p_{1,2},\ldots,p_{1,2k-5}\}$ and
$\mathcal{B}'=\{t,t_1,t_2,p_{1,1}$, $p_{1,2}$, $p_{1,3}\}$, \cref{PropDelta} yields:
 \[D-\deg(\psi(\Lambda))=\sum\nolimits_{b \in \mathcal{B}} s(b),\quad D-\deg(\psi'(\Lambda))=\sum\nolimits_{b \in \mathcal{B}'} s(b).\]
 If $\sum_{b\in \mathcal{B}\setminus \mathcal{B}'} s(b)\le 0$, then $\deg(\psi'(\Lambda))\le \deg(\psi(\Lambda))$ and we are done. Assume thus that $\sum_{b\in \mathcal{B}\setminus \mathcal{B}'} s(b)=\sum_{i=4}^{2k-5} s(p_{1,i})> 0$. Since $m_p\ge m_{p_{1,1}}\ge m_{p_{1,2}}\ge \cdots\ge m_{p_{1,2k-5}}$, we have $s(p_{1,1})\ge s(p_{1,2})\ge \cdots \ge s(m_{p_{1,2k-5}})$. Hence, $0<s(p_{1,4})=m_{p_{1,4}}-\frac{D-m_p}{2}$, so $2m_{p_{1,4}}+m_p>D$. In particular $3m_p\ge 2m_{p_{1,4}}+m_p>D$, so $m_p>D/3$.

 By hypothesis, $m_t+m_{t_1}+m_{t_2}\ge D$, so $s(m_t)+s(m_{t_1})+s(m_{t-2})=m_t+m_{t_1}+m_{t_2}-\frac{3}{2}(D-m_{p})\ge D-\frac{3}{2}(D-m_{p})=\frac{3m_p-D}{2}> 0$. As $s(p_{1,1})\ge s(p_{1,2})\ge s(p_{1,3})\ge s(p_{1,4})>0$, we obtain $\sum_{b\in B} s(b)>0$, and thus $\deg(\psi'(\Lambda))<D$, achieving the proof.
\end{proof}

\subsection{Jonquières maps with a tail}
As we explained in Definition~\ref{d:tail}, an element of $\Jonq(3L)$ has a tail $t \leftarrow t_1 \leftarrow t_2$ if it corresponds to Case \ref{BasePtsJonq2} of \cref{Prop:BasePtsJonq}.
This is the only case of elements of $\Jonq(3L)$ with two proper base points.

\begin{lemma}\label{Lem:Tail4unique}
Let $\phi\in \Bir(\P^2,3L)$ be of degree $4$ with a tail, and assume that the three points of the tail are not collinear. Then, there exist $\rho,\kappa\in A$ such that
\[
\rho\phi\kappa = \chi_1 = \left(-xy^2+y,\frac{1}{y}\right)
\]
is the map of degree $4$ of the relation of Theorem~$\ref{t:main_bis}$.
\end{lemma}

\begin{proof}
As $\deg(\phi)=4$, we are in Case~\ref{BasePtsJonq2} of \cref{Prop:BasePtsJonq} with $m=1$ and $d_1=4$. We denote the base points of $\phi$ by
\[
t\leftarrow t_1 \leftarrow t_2
\quad \text{ and } \quad
p\leftarrow p_1 \leftarrow p_{2}\leftarrow p_3,
\]
where $t_2$ is not proximate to $t$, and $p_1,p_2,p_3$ are proximate to $p$.

We choose $\kappa\in A$ such that the base points of $\phi\kappa$ are
\[
t'=\pcoor{0:1:0} \leftarrow t_1'=\{x=0\}\leftarrow t_2', \quad p'=\pcoor{1:0:0}\leftarrow p_1'=\{y=0\}\leftarrow p_{2}'\leftarrow p_3',
\]
where we give the direction of an irreducible curve to specify an infinitely near point. The point $t_2'$ being not proximate to $t'$ and not collinear with $t'$ and $t_1'$, it corresponds to the direction of $xy-\lambda z^2$ for some $\lambda\in \k^*$. Modifying $\kappa\in A$, we may assume that $\lambda=1$. There exists then $\rho\in A$ such that
\[
\rho\phi\kappa=\chi_1=
\pcoor{(-xy+z^2)y^2:z^4:yz^3}=(-xy^2+y,1/y).
\qedhere
\]
\end{proof}

\begin{remark}
\label{Rem:JonqProdRel}
Let $p\in L$ be a point and let $\phi_1,\dots,\phi_n\in \Bir(\P^2,3L)$ be elements that preserve the pencil of lines through $p$ $($in particular, $\phi_1,\dots,\phi_n\in \Jonq(3L))$.
If $\phi_1\dots \phi_n=\id$, then this relation holds in $\lb \Jonq(3L) \mid R \rb$.
Indeed, if $n\le 3$, the relation belongs to $R$ and if $n\ge 4$, we write $\psi=\phi_1\dots \phi_{n-2}\in \Jonq(3L)$, obtain that $\phi_1\dots \phi_{n-2}\psi^{-1}=\id$ and $\psi \psi_{n-1}\psi_{n}=\id$ and that both relations hold in $\lb \Jonq(3L) \mid R \rb$ by induction, so their product too.
\end{remark}

\begin{lemma}
\label{l:relation_toolboxNew}
For all $\alpha\in \k$, writing $\zeta_\alpha=(x , y+ \frac{\alpha}{x})$, $\chi_\alpha=(-xy^2+\alpha y,\frac{1}{y} )\in \Jonq(3L)$, the relations
\[\chi_\alpha \zeta_\alpha \chi_{-\alpha}=\zeta_\alpha, \quad \zeta_\alpha\zeta_{-\alpha}=\id \text{ and } \chi_{-\alpha}\chi_\alpha=\id\]
hold in the group $\lb \Jonq(3L) \mid R \rb$ from \cref{t:main_bis}.\end{lemma}
\begin{proof}
The relations $\zeta_\alpha\zeta_{-\alpha}=\id$ and $\chi_{-\alpha}\chi_\alpha=\id$ are direct computations and of length $2$, so belongs to $R$. 

Writing $\nu=(x,-y+\frac{1}{x})\in \Jonq(3L)$, the relation of \cref{t:main_bis} is $(\chi_1\nu)^2=\id.$ Writing $\epsilon=(x,-y)\in \Jonq(3L)$, we have $\nu=\zeta_1\epsilon$, $\epsilon\chi_1=\chi_{-1}\epsilon$ and $\epsilon\nu=\zeta_{-1}$.
By \cref{Rem:JonqProdRel}, these three relations hold in $\lb \Jonq(3L) \mid R \rb$, and give
$\nu\chi_1\nu=\zeta_1\epsilon\chi_1\nu=\zeta_1\chi_{-1}\epsilon\nu=\zeta_1\chi_{-1}\zeta_{-1}$ in $\lb \Jonq(3L) \mid R \rb$. Hence, $\chi_1\nu\chi_1\nu=1$ yields $\chi_1\zeta_1\chi_{-1}\zeta_{-1}=1$ in $\lb \Jonq(3L) \mid R \rb$.

This gives the relation $\chi_\alpha \zeta_\alpha \chi_{-\alpha}=\zeta_\alpha$ for $\alpha=1$.
When $\alpha=0$, $\zeta_\alpha=\id$, so the relation holds in $\lb \Jonq(3L) \mid R \rb$ (\cref{Rem:JonqProdRel}).
When $\alpha\in \k\setminus \{0,1\}$, we conjugate $\chi_1 \zeta_1 \chi_{-1}=\zeta_1$ by the linear (hence Jonquières) map $(\alpha x, y)$, and get $\chi_\alpha \zeta_\alpha \chi_{-\alpha}=\zeta_\alpha$.
Again, the conjugation holds in $\lb \Jonq(3L) \mid R \rb$ by \cref{Rem:JonqProdRel}.
\end{proof}

\subsection{Simple maps}\label{Sec:SimpleMaps}
\begin{definition} We say that a birational map $\phi \in \Bir(\P^2,3L)$ of degree $d\ge 2$ is \emph{simple} if it is Jonquières and if there exists lines $L_1,L_2\subset \P^2$, both not equal to $L$, such that $\phi$ restricts to an isomorphism $\P^2\setminus L_1\iso \P^2\setminus L_2$.

Let $d\ge 2$ be an integer. A \emph{simple sequence of points of degree $d$} is a sequence of $2d-1$ points
\[ p\leftarrow p_1 \leftarrow p_2\leftarrow \dots \leftarrow p_{2d-2}\]
where $p\in L$, $p_1$ is not on $L$, $p_1,\dots,p_{d-1}$ are all proximate to $p$, and $p_i$ is only proximate to $p_{i-1}$ for each $i\ge d$.
If $d=2$, we moreover ask that $p,p_1,p_2$ are not collinear (if $d\ge 3$, this follows from the fact that $p_2$ is proximate to $p$).
\end{definition}

\begin{remark}\label{Rem:SimpleBsPts}
Observe that a simple map $\phi\in \Bir(\P^2,3L)$ restricts to a birational map $L\dasharrow L$, which then extends to an isomorphism $L\iso L$. Moreover, in the above definition, $L_1$ is the unique line contracted by $\phi$.

A simple map of degree $d\ge 2$ corresponds to Case~\ref{BasePtsJonq1} of \cref{Prop:BasePtsJonq} with $m=1$ and such that the first infinitely near point $p_1$ is not on $L$.
Its base points are thus a simple sequence of points of degree $d$.
Conversely, every such sequence arises as the base points of a simple map of degree $d\ge 2$, by \cref{Prop:BasePtsJonq}.

\end{remark}

\begin{example}
The birational map
\[
\phi = \left(x + \frac{1}{y^3} + \frac{1}{y^2} + \frac{1}{y}, y\right)
\]
is a simple map in $\Bir(\P^2,3L)$, which contract the line $L_1 = \{y=0\}$.
To see that $\phi$ is an automorphism on $\P^2 \setminus L_1$, one can use homogeneous coordinates and write
\[
\phi = \pcoor{x + \frac{z^4}{y^3} + \frac{z^3}{y^2} + \frac{z^2}{y}: y:z}.
\]

On the other hand, the polynomial automorphism $(x + y^2, y)$ is not a simple map, because the contracted line is equal to $L$.
\end{example}

\begin{propdef}[Successor of two simple maps]\label{p:successor}
Let $\phi,\psi\in \Jonq(3L)$ be simple maps, and assume that the proper base points $p,q\in L$ of $\phi^{-1}$ and $\psi$ respectively are distinct.

For each simple map $\theta\in \Jonq(3L)$, the following conditions are equivalent:
\begin{enumerate}
\item\label{Succ1}
$\theta\psi\phi$ does not contract the unique line contracted by $\phi$
\item\label{Succ2}
$\deg(\theta)=\deg(\phi)$ and the base points of $\theta$ are all base points of $(\psi\phi)^{-1}$, and contain $\psi(p)$.
\item\label{Succ3}
$\psi$ sends the base points of $\phi^{-1}$ onto those of $\theta$.
\end{enumerate}
Moreover, an element $\theta$ that satisfies these conditions is unique up to post-composition with an element of the affine group $A$, and always exists, unless $\deg(\phi)=2$ and $\psi$ sends the three base points of $\phi^{-1}$ onto three collinear points.
The element $\theta$ is called a \emph{successor of $\phi$ and $\psi$}.
\end{propdef}

\begin{proof}
As $\psi(p)$ is not the base point of $\psi^{-1}$, we may (and will) replace $\psi$ with $\alpha\psi$ where $\alpha\in \Aut(\P^2,L)$ and assume that $q$ is the proper base point of $\psi^{-1}$ and that $\psi(p)=p$.

Write $d=\deg(\phi)$ and $r=\deg(\psi)$. The base points of $\phi^{-1}$ and $\psi^{-1}$ are two simple sequences of points (\cref{Rem:SimpleBsPts}), that we write as:
\begin{align*}
p\leftarrow p_1 \leftarrow p_2\leftarrow \dots \leftarrow p_{2d-2}\\
 q\leftarrow q_1 \leftarrow q_2\leftarrow \dots \leftarrow q_{2r-2}\end{align*}
The map $\psi$ restricts to an automorphism $\P^2\setminus L_1\iso \P^2\setminus L_1$ where $L_1$ is the line through $q$ and $q_1$. As $\{q\}=L\cap L_1$, we find $p\in \P^2\setminus L_1$. The automorphism of $\P^2\setminus L_1$ fixes $p$, sends $L$ onto itself and then sends the points $p,p_1,\dots,p_{2d-2}$ onto a set of points
\[p\leftarrow p_1' \leftarrow p_2'\leftarrow \dots \leftarrow p_{2d-2}'\]
which is again a simple sequence of points of degree $d$, unless $d=2$ and the three points $p\leftarrow p_1' \leftarrow p_2'$ are collinear.

The linear system of $\phi^{-1}$ consists of curves of degree $d$ with multiplicity $d-1$ at $p$, and multiplicity $1$ at $p_1,\dots,p_{2d-2}$. The linear system of $(\psi\phi)^{-1}=\phi^{-1}\psi^{-1}$ consists then of curves of degree $dr$ with multiplicity $d(r-1)$ at $q$, multiplicity $d$ at $q_1,\dots,q_{2r-2}$, multiplicity $d-1$ at $p$ and multiplicity $1$ at $p_1',\dots,p_{2d-2}'$.

Condition \ref{Succ3} consists of asking that the base points of $\theta$ are $p,p_1',\dots,p_{2d-2}'$. By \cref{Rem:SimpleBsPts}, such an element is of degree $d$ and exists, except if $d=2$ and the three points are collinear. 
Moreover, if it exists it is unique up to post-composition by an element of $\Aut(\P^2,L)$.

As the base points of a simple map are all infinitely near to a unique proper point, Condition~\ref{Succ2} is equivalent to \ref{Succ3}.

The line $L_2$ contracted by $\phi$ is sent by $\phi$ onto the point $p$, but more precisely on the exceptional divisor associated to $p_{2d-2}$, as this one is the unique $(-1)$-curve in the exceptional divisors obtained by blowing-up $p_1,\dots,p_{2d-2}$, the base points of $\phi^{-1}$.
The map $\psi\phi$ then sends $L_2$ onto the exceptional divisor of $p_{2d-2}'$. Condition \ref{Succ1} then corresponds to asking that $\theta^{-1}$ sends a line onto the exceptional divisor of $p_{2d-2}'$ and is then equivalent to \ref{Succ3}.
\end{proof}

\begin{lemma}
\label{l:successor}
Let $\alpha, \beta, \gamma, \delta \in \k$, and consider the maps
\[
\phi= \left(x, y+ \frac{\delta}{x^3} + \frac{\gamma}{x^2} + \frac{\beta}{x} \right),
\quad
\psi=\left( x+\frac{\alpha}{y},y \right).\]
We assume $(\delta, \gamma, \beta) \neq(0,0,0)$ and $\alpha \neq 0$, so that $\phi$ and $\psi$ are simple maps with $d=\deg(\phi)\in \{2,3,4\}$, and $\deg(\psi)=2$.
Then:
\begin{enumerate}
\item Except if $\delta=\gamma=\alpha+\beta=0$ \parent{which implies that $d=2$}, the successor $\theta_1$ of $\phi$ and $\psi$ exists, and is given by
\[
\theta_1 = \left( x, y-\frac{\delta}{x^3}-\frac{\gamma}{x^2}-\frac{(\beta+(d-1)\alpha)}{x} \right).
\]
\item Except if $\alpha + \beta = 0$, the successor $\theta_2$ of $\psi$ and $\phi$ exists, and is given by
\[
\theta_2 = \left( x - \frac{\alpha + \beta}{y}, y \right).
\]
\end{enumerate}
\end{lemma}

\begin{proof}
Using homogeneous coordinates $\pcoor{x:y:z}$ on $\P^2$, we write
\begin{align*}
\phi&=\pcoor{x : y+ \frac{\delta z^4}{x^3} + \frac{\gamma z^3}{x^2} + \frac{\beta z^2}{x}:z}, &
\psi&=\pcoor{x+\frac{\alpha z^2}{y}:y:z} \\
\theta_1&=\pcoor{x : y - \frac{\delta z^4}{x^3} - \frac{\gamma z^3}{x^2} - \frac{(\beta+(d-1)\alpha) z^2}{x}:z}, &
\theta_2 &=\pcoor{x+\frac{(\alpha + \beta) z^2}{y}:y:z}.
\end{align*}
We compute $\theta_1\psi\phi$ (resp. $\theta_2\phi\psi$), replace $x=0$ (resp. $y = 0$), and check that we obtain $\pcoor{0:uy+vz:z}$ (resp. $\pcoor{ux+vz:0:z}$, where $(u,v)\in \k^2$ and $u \neq 0$.
This shows that $\theta_1$ (resp. $\theta_2$) satisfies Condition~\ref{Succ1} of Proposition~\ref{p:successor}.
\end{proof}

The following two lemmas apply to general Cremona maps, without assuming that a form is preserved. 

\begin{lemma}\label{TheFormulaOfSuccessor}
Let $\phi_1,\phi_2\in \Bir(\P^2)\setminus \Aut(\P^2)$, such that $\phi_1^{-1}$ and $\phi_2$ have no common base points. Assume that there exists $\phi_3\in \Bir(\P^2)$ whose base points are the images by $\phi_{2}$ of the base points of $\phi_{1}^{-1}$, with the same multiplicities.
Consider a resolution of the indeterminacies:
\[
\begin{tikzcd}
&&&Z\ar[dlll,"\eta_0" ',bend right=10]\ar[dl,"\eta_1" ']\ar[dr,"\eta_2"]\ar[drrr,"\eta_3",bend left=10]\\
\P^2 \ar[rr,dashed,"\phi_1"] && \P^2 \ar[rr,dashed,"\phi_2"]&& \P^2 \ar[rr,dashed,"\phi_3"] && \P^2
\end{tikzcd}
\]
Then, writing $d=\deg(\phi_1)=\deg(\phi_3)$, and defining $L\in \mathrm{Pic}(\P^2)$ to be a line of $\P^2$, we have
\[{\eta_3}^* L-{\eta_0}^* L=d\cdot ({\eta_2}^* L- {\eta_1}^* L).\]
\end{lemma}

\begin{proof}
We have ${\eta_0}^* L=d {\eta_1}^* L-R$, where $R\in \mathrm{Pic}(Z)$ is an effective divisor that is the sum of the pullbacks of the base points of $(\phi_1)^{-1}$ with multiplicities. By assumption, these are the same points, via $\phi_2$, as the base points of $\phi_3$, with the same multiplicities. We then obtain
$R=d {\eta_1}^* L-{\eta_0}^* L=d{\eta_2}^* L-{\eta_3}^* L$,
which gives the result.
\end{proof}

\begin{lemma}
\label{l:degree_sequence}
Let $\phi_1,\phi_2\in \Bir(\P^2)\setminus \Aut(\P^2)$, such that $\phi_1^{-1}$ and $\phi_2$ have no common base points. Assume that for each $i\ge 3$, there exists $\phi_i\in \Bir(\P^2)$ whose base points are the images by $\phi_{i-1}$ of the base points of $\phi_{i-2}^{-1}$, with the same multiplicities.

We write $d_i=\deg(\phi_i\dots \phi_1)$ for each $i\ge 1$, and by convention $d_0=1$. Then, for each $i\ge 3$ we have
\[d_i=\deg(\phi_j) (d_{i-1}-d_{i-2})+d_{i-3},\]
where $j\in \{1,2\}$ has the same parity as $i$.

In particular, writing $a=\deg(\phi_1)$ and $b=\deg(\phi_2)$, the sequence $(d_i)_{i\ge 0}$ is bounded if and only if
\[
(a,b) \in \{ (2,2), (2,3), (3,2), (2,4), (4,2)\}.
\]
In these cases, the sequence is in fact periodic:
\begin{itemize}
\item If $(a,b) = (2,2)$, $(d_i)_{i\ge 0} = (1,2,4,5,4,2,\dots)$ of period $6$.
\item If $(a,b) = (2,3)$, $(d_i)_{i\ge 0} = (1, 2, 6, 9, 11, 10, 6, 3,\dots)$ of period $8$.
\item If $(a,b) = (3,2)$, $(d_i)_{i\ge 0} = (1, 3, 6, 10, 11, 9, 6, 2,\dots)$ of period $8$.
\item If $(a,b) = (2,4)$, $(d_i)_{i\ge 0} = (1, 2, 8, 13, 22, 26, 29, 28, 22, 17, 8, 4,\dots)$ of period $12$.
\item If $(a,b) = (4,2)$, $(d_i)_{i\ge 0} = (1, 4, 8, 17, 22, 28, 29, 26, 22, 13, 8, 2,\dots)$ of period $12$.
\end{itemize}
\end{lemma}
\begin{proof}
We write $a=\deg(\phi_1)$, $b=\deg(\phi_2)$.
As $\phi_1^{-1}$ and $\phi_2$ have no common base points we have $\deg(\phi_2\phi_1)=ab$ and the base points of $\phi_2\phi_1$ are the base points of $\phi_1$ and the points sent by $\phi_1$ onto the base point of $\phi_2$. We denote by $\eta_0\colon Z\to \P^2$ the blow-up of the base points of $\phi_2\phi_1$. Hence, $\eta_2=\phi_2\phi_1\eta_0$ is a birational morphism. Moreover, the base points of $\phi_1$ are blown-up by $\eta_0$, so $\eta_1=\phi_1\eta_0$ is a birational morphism.

For each $i\ge 3$, the base points of $\phi_i$ are contained in the base points of $(\phi_{i-1}\dots \phi_1)^{-1}$ and thus $\eta_i=\phi_i\dots \phi_1\eta_0$ is a birational morphism from $Z$ to $\P^2$. This gives the following commutative diagram:
\[
\begin{tikzcd}
&& Z\ar[dll,"\eta_0" ',bend right=12]\ar[dl,"\eta_1" ']\ar[d,"\eta_2" ]\ar[dr,"\eta_3"]\ar[drrr,"\eta_i",bend left=10]\\
\ar[r,dashed,"\phi_1" ']\P^2 &
\ar[r,dashed,"\phi_2" '] \P^2 &
\ar[r,dashed,"\phi_3" ']\P^2 &
\ar[r,dashed,"\phi_4" '] \P^2 &
\ar[r,dashed,"\phi_i" ']  \cdots &
\ar[r,dashed,"\phi_{i+1}" '] \P^2 &
\cdots
\end{tikzcd}
\]

For each $i\ge 1$, the degree of $\phi_i\dots \phi_1$ is given by $d_i={\eta_i}^* L\cdot {\eta_0}^* L.$ We moreover have $d_1=a$, $d_2=ab$.

By \cref{TheFormulaOfSuccessor}, for $i\ge 2$ we get
\begin{align*}
{\eta_{2i-1}}^* L&= a({\eta_{2i-2}}^* L-{\eta_{2i-3}}^* L)+{\eta_{2i-4}}^* L, \\
{\eta_{2i}}^* L &= b({\eta_{2i-1}}^* L-{\eta_{2i-2}}^* L)+{\eta_{2i-3}}^* L.
\end{align*}
This gives
\begin{align*}
d_{2i-1}&= a(d_{2i-2}-d_{2i-3})+d_{2i-4}, \\
d_{2i} &= b (d_{2i-1}-d_{2i-2})+d_{2i-3},
\end{align*}
and yields the stated formula. To decide when it is bounded, we observe that for $i\ge 3$, we have
\[
\begin{psmallmatrix} d_i\\ d_{i-1} \\ d_{i-2}\end{psmallmatrix}=M_i\begin{psmallmatrix} d_{i-1}\\ d_{i-2} \\ d_{i-3}\end{psmallmatrix}\text{ where }M_i=\begin{psmallmatrix} \deg(\phi_j) & -\deg(\phi_j) & 1 \\ 1 & 0 & 0 \\ 0 & 1 & 0\end{psmallmatrix},
\]
and $j\in \{1,2\}$ has the same parity as $i$. The characteristic polynomial of $M_2M_1$ is $(X-1)(X^2-\mu X+1)$ where $\mu=ab-a-b-1=(a-1)(b-1)-2\ge -1$. If $\mu=2$, then $1$ is the only eigenvalue and as $M_2M_1$ is not the identity, the growth of $(d_i)_{i\ge 1}$ is linear. If $\mu \ge 3$, there is a real eigenvalue $\lambda>1$ and an eigenvalue $\frac{1}{\lambda}<1$, the growth of $(d_i)_{i\ge 1}$ is then exponential.
If $\mu<2$, we only have $(a,b) \in \{ (2,2), (2,3), (3,2), (2,4), (4,2)\}$, and the explicit computation of the sequence in each case is direct, which shows that the sequence $(d_i)_{i\ge 1}$ is bounded in these cases.
\end{proof}

\begin{lemma}
\label{l:simple_relation}
Let $\phi_1,\phi_2\in \Jonq(3L)$ be two simple maps of degree $2$ and $d\in \{2,3,4\}$ respectively.
Assume that there is no member of the linear system of $\phi_2$ passing through the three base points of $\phi_1^{-1}$.
If $d=4$, assume moreover that there is no quintic having multiplicity $3$ at the maximal base point of $\phi_2$, passing through all base points of $\phi_2$, and having multiplicity $2$ at the three base points of $\phi_1^{-1}$.

Then for $3 \le i \le n=n_d$, where $n_2=6$, $n_3=8$, $n_4=12$, we can choose a successor $\phi_i$ of $\phi_{i-2},\phi_{i-1}$ such that
\[
\phi_n\dots \phi_1=\id.
\]
Moreover, any such relation holds in the group $\lb \Jonq(3L) \mid R \rb$ from \cref{t:main_bis}.
\end{lemma}

\begin{proof}

The base points of $\phi_1^{-1}$ and $\phi_2$ are two simple sequences of points that we denote respectively by
\[
p\leftarrow p_1 \leftarrow p_2
\quad \text{ and } \quad
q\leftarrow q_1 \leftarrow q_{2}\leftarrow \dots \leftarrow q_{2d-2},
\]
where by assumption $p \neq q$.
We may replace $\phi_1$ and $\phi_2$ with $\kappa\phi_1$ and $\phi_2\kappa^{-1}$ for some $\kappa\in A$. Doing this, we may choose
\[
p=\pcoor{0:1:0} \leftarrow p_1=\{x=0\}, \quad q=\pcoor{1:0:0}\leftarrow q_1=\{y=0\},
\]
where we give the direction of an irreducible curve to specify an infinitely near point.

Replacing $\phi_1$ and $\phi_2$ with $\phi_1\tau$ and $\rho\phi_2$ for some $\tau,\rho\in A$, we obtain
\[
\phi_1=\left(x,y+\frac{\alpha_1}{x}\right), \quad
\phi_2=\left(x+\frac{\beta_1}{y}+\frac{\gamma_1}{y^2}+\frac{\delta_1}{y^3},y\right),
\]
for some $\alpha_1,\beta_1,\gamma_1,\delta_1 \in \k$.
We now apply \cref{l:successor} to calculate the successors
\[
\phi_{2i-1}=\left(x,y+\frac{\alpha_i}{x}\right)\text{ and }
\phi_{2i}=\left(x+\frac{\beta_i}{y}+\frac{\gamma_i}{y^2}+\frac{\delta_i}{y^3},y\right)
\] for each $i\in \{0,\ldots,\frac{n}{2}\}$. For each $i\ge 1$ we find
\[
\gamma_i=(-1)^{i+1}\gamma_1,\quad \delta_i=(-1)^{i+1}\delta_1, \quad \begin{psmallmatrix}
\alpha_i\\
\beta_i\end{psmallmatrix}=\begin{psmallmatrix}
-1 & -1\\
d-1& d-2\end{psmallmatrix}\begin{psmallmatrix}
\alpha_{i-1}\\
\beta_{i-1}\end{psmallmatrix},\]
which yields
\[
\begin{tabular}{CCC}
\toprule
d=2& d=3 & d=4\\
\midrule
\begin{psmallmatrix}
\alpha_{1}\\
\beta_{1}\\
\alpha_{2}\\
\beta_{2}\\
\alpha_{3}\\
\beta_{3}
\end{psmallmatrix}=
\begin{psmallmatrix*}[r]
1 &0\\
0 & 1\\
-1 & -1\\
1& 0\\
0 & 1\\
-1 & -1\end{psmallmatrix*}\begin{psmallmatrix}
\alpha_{1}\\
\beta_{1}\end{psmallmatrix} & \begin{psmallmatrix}
\alpha_{1}\\
\beta_{1}\\
\alpha_{2}\\
\beta_{2}\\
\alpha_{3}\\
\beta_{3}\\
\alpha_{4}\\
\beta_{4}
\end{psmallmatrix}=
\begin{psmallmatrix*}[r]
1 &0\\
0 & 1\\
-1 & -1\\
2& 1\\
-1& 0\\
0 & -1\\
1 & 1\\
-2 & -1\end{psmallmatrix*}\begin{psmallmatrix}
\alpha_{1}\\
\beta_{1}\end{psmallmatrix}&\begin{psmallmatrix}
\alpha_{1}\\
\beta_{1}\\
\alpha_{2}\\
\beta_{2}\\
\alpha_{3}\\
\beta_{3}\\
\alpha_{4}\\
\beta_{4}\\
\alpha_{5}\\
\beta_{5}\\
\alpha_{6}\\
\beta_{6}
\end{psmallmatrix}=
\begin{psmallmatrix*}[r]
1 &0\\
0 & 1\\
-1 & -1\\
3& 2\\
-2& -1\\
3 & 1\\
-1 & 0\\
0 & -1\\
1 & 1\\
-3 & -2 \\
2 & 1\\
-3&-1
\end{psmallmatrix*}\begin{psmallmatrix}
\alpha_{1}\\
\beta_{1}\end{psmallmatrix}\\
\bottomrule\\
\end{tabular}
\]

To see that the successors exist, we need to check that $-\alpha_{i+1}=\alpha_i+\beta_i$ is not zero for each $i$, and if $d=2$, we also need that $-\beta_{i+1}=\beta_i+\alpha_{i+1}$ is not zero (see \cref{l:successor}). For this, we need $\alpha_1+\beta_1\not=0$ in each case, and $2\alpha_1+\beta_1\not=0$ in the case $d=4$.
We now prove that this corresponds to the geometric hypothesis we requested.

As $p_2$ is a base point of $\phi_1^{-1}=(x,y-\frac{\alpha_1}{x})$, it corresponds to the direction of the conic with equation $xy-\alpha_1 z^2$.

The unique element of the linear system of $\phi_2$ passing through $p$ and $p_1$ is $y^{d-1}(x+\frac{\beta_1z^2}{y}+\frac{\gamma_1z^3}{y^2}+\frac{\delta_1z^4}{y^3})=0$. It then passes through $p_2$ if and only if $\alpha_1+\beta_1=0$. We thus obtain $\alpha_1+\beta_1\not=0$.

In the case $d=4$, the fact that $\alpha_1+\beta_1\not=0$ yields the existence of $\phi_3=(x,y+\frac{\alpha_2}{x})$.
Then, $\phi_4=(x+\frac{\beta_2}{y}+\frac{\gamma_2}{y^2}+\frac{\delta_2}{y^3},y)$ is a successor.
The existence of the successor $\phi_5$ of $\phi_3$ and $\phi_4$ is then equivalent to the condition $\alpha_2+\beta_2\not=0$, and thus to $2\alpha_1+\beta_1\not=0$.
As we saw before, if $\alpha_2+\beta_2=0$, there exists an element of the linear system of $\phi_4$ passing through the base points of $\phi_3^{-1}$.
The preimage by $\phi_3$ of such a curve will be a quintic passing through the base points of $\phi_3$ with multiplicity $2$, passing through all the base points of $\phi_2^{-1}$, and with multiplicity $3$ at the proper base point.
The preimage by $\phi_2$ is then a quintic having multiplicity $3$ at the maximal base point of $\phi_2$, passing through all base points of $\phi_2$, and having multiplicity $2$ at the three base points of $\phi_1^{-1}$. By assumption, such a curve does not exist, so $2\alpha_1+\beta_1\not=0$.

We know from \cref{l:degree_sequence} that $\phi_n \dots \phi_1$ has degree $1$, and thus corresponds to an element of $A$. Actually, we can replace the $\phi_i$ with the formulas above and check that it is the identity.
It remains to see that the relation
\[\phi_n \dots \phi_1=\id_{\P^2}\]
holds in $\lb \Jonq(3L) \mid R \rb$.

To do that, we insert some $\chi_{\pm\alpha_i}$ before and after each $\phi_{2i-1}=\zeta_{\alpha_i}$: we use the relation $\chi_{\alpha_i}\zeta_{\alpha_i}\chi_{-\alpha_i}=\zeta_{\alpha_i}$ (\cref{l:relation_toolboxNew}), and then insert $\chi_{\alpha_i}$ before $\phi_{2i}$ and  $\chi_{-\alpha_i}$ after $\phi_{2i-2}$ (if $i=1$, this insert $\chi_{-\alpha_1}$ after $\phi_n$).

The $\phi_j$ with $j$ odd do not change. The $\phi_{2i}=(x+\frac{\beta_i}{y}+\frac{\gamma_i}{y^2}+\frac{\delta_i}{y^3},y)$ are replaced with
\[\phi_{2i}'=\chi_{\alpha_i+\beta_i} \phi_{2i} \chi_{\alpha_i}=\left(x-\gamma_i-\delta_i y,y \right)\in A,\]
and this equality holds in $\lb \Jonq(3L) \mid R \rb$, as all maps preserve the pencil through $\pcoor{1:0:0}$ (\cref{Rem:JonqProdRel}). The new relation is then
\[ \phi_{n}'\phi_{n-1}\dots \phi_2'\phi_1=\id.\]

If $d\le 3$, the $\phi_{2i}'$ (and the $\phi_{2i-1}$) are all preserving the pencil of lines through $\pcoor{0:1:0}$ (lines $x=\mathrm{cst}$),
so $\phi_{n}'\phi_{n-1}\dots \phi_2'\phi_1=\id$ holds in $\lb \Jonq(3L) \mid R \rb$ (\cref{Rem:JonqProdRel}).

Suppose now that $d=4$, hence $n=12$. For each $i\in \{2,4,6\}$, we consider
\[
\theta_i = \phi_{2i}' \phi_{2i -1}\phi_{2i-2}'=\phi_{2i}' \phi_{2i -1}\phi_{2i}'^{-1}=(x+\gamma_1+\delta_1 y,y)\left(x,y+\frac{\alpha_i}{x}\right)(x-\gamma_1-\delta_1 y,y),
\]
which is a simple Jonquières element of degree $2$ conjugate by an element of $A$, and thus is still a simple Jonquières element of degree $2$. Moreover, the above relation is generated by trivial relations of length $3$. Our relation is then equivalent to the relation
\[ \theta_6\phi_9\theta_4\phi_5\theta_2\phi_1=\id.\]
We now conjugate this relation with $\nu=(x,y-x/\delta_1+\gamma_1/\delta_1)\in A$. This replaces $\phi_i$ $(i\in \{1,5,9\})$ and $\theta_i$ $(i\in \{2,4,6\}$) with
\[
\nu \phi_i \nu^{-1}=\left( x,y+\frac{\alpha_i}{x} \right),
\quad
\nu \theta_i\nu^{-1}=\left(x-\frac{\alpha_i }{y},y\right).
\]
This corresponds to a relation as in the case $d=2$.
\end{proof}

\section{Decreasing the degrees of linear systems using relations}
\label{Sec:ProofBis}
In this section we study the situation where two distinct Jonquières maps decrease the degree of a given homaloidal system.
This is the last needed ingredient for the proof of the presentation of \cref{t:main_bis}.

\subsection{Peak configurations}

\begin{definition}
\label{d:peak}
Let $\Lambda$ be a homaloidal linear system on $\P^2$, and let $\phi, \phi' \in \Jonq(3L)\setminus A$.
We say that $(\phi, \phi', \Lambda)$ is a \emph{peak configuration} if \[\deg(\Lambda) \ge \deg(\phi(\Lambda)), \quad\deg(\Lambda) \ge \deg(\phi'(\Lambda)),\] with at least one of these inequalities being strict.

We say that $(\phi, \phi', \Lambda)$ fits into a \emph{degree lowering relation of length $n\ge 2$} if there exists $\phi_1,\ldots,\phi_n\in \Jonq(3L)$ such that $\phi_n\dots \phi_1=\id$ holds in the group $\lb \Jonq(3L) \mid R \rb$ from \cref{t:main_bis}, such that $\phi=\phi_1, \phi'=(\phi_n)^{-1}$ and for each $i \in \{ 2, \dots, n-2\}$
\[
 \deg (\phi_i \dots \phi_1 (\Lambda)) < \deg (\Lambda).
\]
\[
\begin{tikzcd}[dashed]
&& \Lambda \ar[drr,"\phi_n^{-1} = \phi'"] \ar[dll,"\phi = \phi_1" ']\\
\phi(\Lambda) \ar[d,"\phi_2"] &&&& \phi'(\Lambda) \\
\phi_2\phi_1(\Lambda) \ar[r,"\phi_3" '] &
\cdots \ar[r,"\phi_i" '] &
\phi_i\dots\phi_1(\Lambda) \ar[r,"\phi_{i+1}" ']&
\cdots \ar[r,"\phi_{n-2}" '] &
\phi_{n-2}\dots\phi_1(\Lambda) \ar[u,"\phi_{n-1}"]
\end{tikzcd}
\]

\end{definition}
\begin{remark}\label{Rem:ExchangePhiPhiprim}
Let $(\phi, \phi', \Lambda)$ be a peak configuration. Then, $(\phi', \phi, \Lambda)$ is again a peak configuration. Moreover, the first one fits into a degree lowering relation if and only if the second one does. We may thus often exchange $\phi$ and $\phi'$.
\end{remark}

Peak configurations will be used to show that we have indeed found all the relations in $\Bir (\P^2, 3L)$ in \S \ref{s:volume preserving-jonquieres-maps}.
We now prove a series of lemmas that will be useful in the proof of \cref{t:main_bis}.
Peak configurations are in fact quite restrained and are related to the relations in $R$.

\begin{lemma}\label{l:peak_tail}
Let $(\phi, \phi', \Lambda)$ be a peak configuration.
Up to exchanging $\phi$ and $\phi'$, we assume that $\phi$ is of degree $4$ with a tail
\parent{Case~\ref{BasePtsJonq2} of \cref{Prop:BasePtsJonq}}, and that
$\phi'$ is of degree $2$ with base points at the three points $t,t_1,t_2$ of the tail of $\phi$.
Then, $(\phi, \phi', \Lambda)$ fits into a degree lowering relation of length $4$.
\end{lemma}
\begin{proof}
By \cref{Lem:Tail4unique}, there exist $\rho,\kappa\in A$ such that
$\rho\phi\kappa=\chi_1=(-xy^2+y,\frac{1}{y})$.
We may thus replace $(\phi, \phi', \Lambda)$ with $(\rho\phi\kappa,\phi'\kappa, \kappa^{-1}(\Lambda))$ and assume that
\[
\phi=\chi_1=\pcoor{(-xy+z^2)y^2:z^4:yz^3}=(-xy^2+y,1/y).
\]
This being done, the base points of $\phi'$ are those of
\[\zeta_{-1}=\pcoor{x^2:xy-z^2:xz}=(x , y-1/x).\]
There exists thus $\rho'\in A$ such that $\rho'\phi'=\zeta_{-1}$. We replace $\phi'$ with $\zeta_{-1}$ and have reduced to the case where $\phi=\chi_1, \phi'=\zeta_{-1}$. 
We then use the relation
\[\zeta_{1}\chi_{-1}\zeta_{-1}\chi_1=\id.
\]
from \cref{l:relation_toolboxNew} (with $\alpha=-1$).
To see that this is a degree lowering relation of length $4$, we need to show that \[\deg(\zeta_{-1}\chi_1(\Lambda))<\deg(\Lambda).\]
The base points of $\phi=\chi_1$ are
\[
t=\pcoor{0:1:0} \leftarrow t_1\leftarrow t_2, \quad p=\pcoor{1:0:0}\leftarrow p_1\leftarrow p_{2}\leftarrow p_3,
\]
and $p$ it the base point of multiplicity $3$. Moreover, the base points of $\phi'$ are $t,t_1,t_2$.
Writing $d=\deg(\Lambda)$, by \cref{PropDelta} we obtain:
\begin{align*}
\deg(\phi(\Lambda))-d&=3(d-m_p(\Lambda))-\sum\nolimits_{i=1}^2 m_{p_i}(\Lambda)-m_t(\Lambda)-\sum\nolimits_{i=1}^3 m_{t_i}(\Lambda),\\
\deg(\phi'(\Lambda))-d&=(d-m_t(\Lambda))-\sum\nolimits_{i=1}^2 m_{t_i}(\Lambda).
\end{align*}
The base points of
\[
\zeta_{-1}\chi_1 = \left(y(1-xy),\frac{x}{xy-1} \right)=
\pcoor{-y(xy-z^2)^2:xz^4:(xy-z^2)z^3}
\] are the same as those of $\phi$, but with $t, t_1,t_2$ of multiplicity $2$. This gives
\begin{align*}
\deg(\zeta_{-1}\chi_1(\Lambda))&=\deg(\phi(\Lambda))+ \deg(\phi'(\Lambda))-d\\
&=d+(\deg(\phi(\Lambda))-d)+(\deg(\phi'(\Lambda))-d)<d.
\qedhere
\end{align*}
\end{proof}

\begin{lemma}
\label{l:simple_peak}
Let $(\phi, \phi', \Lambda)$ be a peak configuration,
where $\phi$ and $\phi'$ are simple maps with distinct proper points.
Up to exchanging $\phi$ and $\phi'$, we assume that $r=\deg(\phi) \in \{2,3,4\}$, and $\deg(\phi')=2$.
Then, $(\phi, \phi', \Lambda)$ fits into a degree lowering relation of length $n_r$, with $n_2=6, n_3=8, n_4=12$.
This relation is obtained by taking successors as in \cref{l:simple_relation}.
\end{lemma}

\begin{proof}
The base points of $\phi'$ and $\phi$ are two simple sequences of points that we denote respectively by
\[
p\leftarrow p_1 \leftarrow p_2 \quad \text{ and }q\leftarrow q_1 \leftarrow q_{2}\leftarrow \cdots \leftarrow q_{2r-2},
\]
where by assumption $p \neq q$.

Writing $d=\deg(\Lambda)$ and denoting by $m_u$ the multiplicity of $\Lambda$ at a point $u$, the inequalities $\deg(\Lambda) \ge \deg(\phi(\Lambda)),\deg(\Lambda) \ge \deg(\phi'(\Lambda))$ correspond to (\cref{PropDelta}):
\[ m_p+m_{p_1}+m_{p_2}\ge d, \quad (r-1)m_q+\sum_{i=1}^{2r-2} m_{q_i}\ge (r-1)d.\]
We now prove that these two inequalities, and the fact that one is strict, imply that the hypotheses of
 \cref{l:simple_relation} are satisfied. Assume first that a member of the linear system of $\phi$ passes through the three base points of $\phi'$. The intersection with a general member of $\Lambda$ would then be at least $(r-1)m_q+\sum_{i=1}^{2r-2} m_{q_i}+m_p+m_{p_1}+m_{p_2}>rd$, which is impossible by Bézout's theorem.
Similarly, if $r=4$, there is no quintic having multiplicity $3$ at $q$, passing through all base points of $\phi$, and having multiplicity $2$ at the three base points of $\phi'$. Indeed, otherwise the intersection with a general member of $\Lambda$ would be at least $3m_q+\sum_{i=1}^{2r-2} m_{q_i}+2(m_p+m_{p_1}+m_{p_2})>5d$, again contradicting Bézout's theorem.

We define $\phi_1=\phi$, then define $\phi_2$ to be the successor of $\phi_1$ and $\phi'^{-1}$, then define, for each $i\in \{3,\ldots,n-1\}$ the element $\phi_{i}$ to be the successor of $\phi_{i-1}$ and $\phi_{i-2}$. By \cref{l:simple_relation}, we can choose these to satisfy
$\phi_{n-1}\dots \phi_1\phi'^{-1}=\id,$ with a relation that holds in $\lb \Jonq(3L) \mid R \rb$.
Choosing $\phi_n=\phi'^{-1}$ we then obtain
\[
\phi_{n}\dots \phi_1=\id.
\]
It remains to see that
\[
\label{loweri}\tag{$\spadesuit$}
 \deg (\phi_i \dots \phi_1 (\Lambda)) < \deg (\Lambda)
\]
for each $i \in \{ 2, \dots, n-2\}$.

We take a common resolution of $\phi_1,\dots,\phi_n$:
\[
\begin{tikzcd}[row sep = 15]
&&Z\ar[d,"\eta_0" ']\ar[dl,"\eta_{n-1}" ',bend right=10]\ar[dr,"\eta_1" ']\ar[drr,"\eta_2" ',bend left=5]\ar[drrr,"\eta_3",bend left=10]\\
\cdots\ar[r,dashed,"\phi_{n-1}" '] &
\P^2\ar[r,dashed,"\phi_n" '] &
\ar[r,dashed,"\phi=\phi_1" ']\P^2 &
\ar[r,dashed,"\phi_2" '] \P^2 &
\ar[r,dashed,"\phi_3" ']\P^2 &
\ar[r,dashed,"\phi_4" '] \P^2 &
\cdots
\end{tikzcd}
\]
Writing $L$ a line of $\P^2$, we now prove that for each $i\in \{0,1,\ldots, n-1\}$, there exist $\alpha_i,\beta_i\ge 0$ such that
\[
\tag{$\ast$}\label{etaialphaibetai}
\eta_i^*L-\eta_0^*L= \alpha_i (\eta_1^*L-\eta_0^*L)+\beta_i (\eta_{n-1}^*L-\eta_0^*L).
\]
For $i\in \{n-1,0,1\}$, one has $(\alpha_0,\beta_0)=(0,0)$, $(\alpha_1,\beta_1)=(1,0)$ and $(\alpha_{n-1},\beta_{n-1})=(0,1)$.

For $i\in \{2,\ldots,n-2\}$, we apply \cref{TheFormulaOfSuccessor}, that gives
${\eta_i}^* L-{\eta_{i-3}}^* L=\deg(\phi_i)\cdot ({\eta_{i-1}}^* L- {\eta_{i-2}}^* L)$, and thus yields
\[{\eta_i}^* L-\eta_0^*L=\deg(\phi_i)\cdot ({\eta_{i-1}}^* L-\eta_0^*L)-\deg(\phi_i)( {\eta_{i-2}}^* L-\eta_0^*L)+({\eta_{i-3}}^* L-\eta_0^*L).\]
Proceeding by induction, this gives \eqref{etaialphaibetai}, for some integers $\alpha_i,\beta_i$.
We now prove $\alpha_i\ge 0$ and $\beta_i\ge 0$. Let $q$ be a base point of $\varphi$, and let $E_q\in \mathrm{Pic}(Z)$ be the pull-back of its exceptional divisor. The three integer $E_q\cdot \eta_{i}^*L$, $E_q\cdot \eta_1^*L$ and $E_q\cdot\eta_{n-1}^*L$ are the multiplicities of $\varphi_i\dots \varphi_1$, $\phi$ and $\phi'$ at $q$ respectively. These numbers are nonnegative, equal to $0$ for $\phi'$, and $>0$ for $\phi$.
As $E_q\cdot \eta_0^*L=0$, intersecting $E_q$ with both sides of \eqref{etaialphaibetai} proves that $\alpha_i\ge 0$. Doing the same with a base point of $\phi'$ yields $\beta_i\ge 0$.

We now choose $i\in \{2,\ldots,n-2\}$ and use $\deg (\phi_i \dots \phi_1 (\Lambda))=\phi_i \dots \phi_1 (\Lambda)\cdot L={\eta_i}^*(\phi_i \dots \phi_1 (\Lambda))\cdot {\eta_i}^*L={\eta_0}^* \Lambda\cdot {\eta_i}^* L $ to get
\[\deg (\phi_i \dots \phi_1 (\Lambda))-\deg(\Lambda)=\alpha_i
\underset{\le 0}{\underbrace{(\deg(\phi (\Lambda))-\deg(\Lambda))}}+\beta_i
\underset{\le 0}{\underbrace{(\deg(\phi' (\Lambda))-\deg(\Lambda))}}\le 0.\] This proves~\eqref{loweri}, unless we have an equality.
As one of the inequalities $\deg(\phi(\Lambda))\le \deg(\Lambda)$ and $\deg(\phi'(\Lambda))\le \deg(\Lambda)$ is strict, the only possibility to have an equality here is when $\alpha_i=0$ or $\beta_i=0$.
If $\beta_i=0$, we obtain $\eta_i^*L-\eta_0^*L=\alpha_i (\eta_1^*L-\eta_0^*L)$.
The multiplicities of $\phi_i \dots \phi_1$ are thus multiple of the multiplicities of $\phi_1$.
As $\phi_i \dots \phi_1$ is not the identity and does not have the same degree as $\phi_1$ (\cref{l:degree_sequence}), this is impossible, by Noether's formula (\cref{l:noether_rel}). The case $\alpha_i=0$ is symmetrical.
\end{proof}

\begin{lemma}\label{lemme:low-degree-tail-implies-large-m0}
Let $\Lambda$ be a homaloidal linear system of degree $d$.
Let $\phi\in \Jonq(3L)$ with maximal base point $p$ and a tail $t \leftarrow t_1 \leftarrow t_2$, such that $m_t+m_{t_1}+m_{t_1}\le d$, where $m_u=m_u(\Lambda)$ is the multiplicity of $\Lambda$ at a point $u$.
Then, \[\deg (\phi(\Lambda)) \leq d\Rightarrow m_p\ge d/2\quad\text{ and }\quad\deg (\phi(\Lambda)) < d\Rightarrow m_p> d/2.\]
\end{lemma}

\begin{proof}
We denote by $k$ the degree of $\phi$ and use the notation of Case~\ref{BasePtsJonq2} of \cref{Prop:BasePtsJonq} for the base points of $\phi$ and their numbering. 
We then obtain by \cref{{PropDelta}}
 \[
\label{differencetaillarge}
\tag{$\heartsuit$}
 d-\deg (\phi (\Lambda)) =(k-1) (m_p-d) + \sum_{i,j} m_{p_{i,j}} + m_t+m_{t_1}+m_{t_2}.
 \]
 For each $i\in \{1,\ldots,m\}$, the points $p_{i,1}, \dots, p_{i, d_i - 1}$ are proximate to $p$, so we obtain $\sum_{i = 1}^m \sum_{j = 1}^{d_i -1} m_{p_{i,j}}(\Lambda) \leq m_p(\Lambda)$. These are $\sum_{i=1}^{m} (d_i-1)\stackrel{\ref{Prop:BasePtsJonq}}{=}k-1$ points.
 The points $t,t_1,t_2$ satisfy $m_t+m_{t_1}+m_{t_2}\le d$ by hypothesis. The remaining points $k-4$ points have multiplicity at most $d-m_p(\Lambda)$ by \cref{TwoptsHom}. Replacing these in \eqref{differencetaillarge} yields
 \[d-\deg (\phi (\Lambda)) \le (k-1) (m_p-d) + m_p+(k-4)(d-m_p) + d\le 4m_p-2d,\]
 hence the result.
\end{proof}

\begin{lemma}
\label{l:no_common_points}
Let $(\phi,\phi',\Lambda)$ be a peak configuration.
Assume $\phi$ and $\phi'$ have no common base points, and $\deg(\phi) \ge \deg(\phi')$.
Then:
\begin{enumerate}
\item\label{kprime2}
$\deg(\phi')= 2$ and $\deg(\phi) \in \{2,3,4\}$.
\item\label{Notaildisjoint}
The map $\phi$ has no tail. In particular, it corresponds to Case~\ref{BasePtsJonq1} of \cref{Prop:BasePtsJonq}.
\item \label{no2coll}
If $\phi$ has only one base point in the neighbourhood of its maximal base point
\parent{$m=1$ in the notation of \cref{Prop:BasePtsJonq}\ref{BasePtsJonq1}},
then $\phi$ and $\phi'$ are simple.
\end{enumerate}
\end{lemma}

\begin{proof}
We write $d=\deg(\Lambda)$, $k=\deg(\phi)$, $k'=\deg(\phi')$, denote by $B$ (resp.~$B'$) the set of base points of $\phi$ (resp.~$\phi'$), and write $m_u=m_u(\Lambda)$ the multiplicity of a point $u$. We also denote by $p$ and $p'$ the maximal base points of $\phi$ and $\phi'$ respectively.

\begin{enumerate}
\item[\ref{kprime2}]
Since $\deg\phi(\Lambda)\le d$ and $\deg\phi'(\Lambda)\le d$, by \cref{PropDelta} we obtain:
\[
\sum_{u\in B\setminus \{p\}} m_u\ge (d-m_p)(k-1),
\quad
\sum_{u\in B'\setminus\{p'\}} m_u\ge (d-m_{p'})(k'-1),
\]
and one of the two inequalities is strict. By Noether's formula \ref{l:noether_rel} we have $\sum_{u\in B} m_u + \sum_{u\in B'} m_u \le 3d - 3$. Adding $m_p+m_{p'}$ to the two above inequalities yields
\begin{equation}\label{ppprime}
\tag{$\diamondsuit$}
3d-3>(d-m_{p}-m_{p'})(k-2)+dk'+(k-k')m_{p'}.
\end{equation}
As $k\ge k'\ge 2$ and $d\ge m_p+m_p'$ (Lemma~\ref{TwoptsHom}), we obtain $3d-3> dk'$, hence $k'=2$, proving the first part of~\ref{kprime2}.

Replacing $k'=2$ in \eqref{ppprime}, we obtain $3d-3>(d-m_p)(k-2)+2d=k d -(k -2) m_{p}$, so $m_{p} > \frac{(k -3) d + 3}{k -2}.$

Every $u\in B'$ satisfies $m_u+m_p\le d$ (Lemma~\ref{TwoptsHom}), and thus $m_u< d- \frac{(k -3) d + 3}{k -2}= \frac{d-3}{k -2}.$

As $\deg\phi'(\Lambda)\le d$, there exist $u_1,u_2\in \Lambda\setminus \{p'\}$ such that $m_{p'}+m_{u_1}+m_{u_2}\ge d$ (\cref{PropDelta}). This gives $d\le 3\frac{d-3}{k -2}$, which implies that $k\in \{2,3,4\}$ and achieves the proof of \ref{kprime2}.

\item[\ref{Notaildisjoint}]
Assume for contradiction that $\phi$ has a tail, i.e. we are in Case \ref{BasePtsJonq2} of \cref{Prop:BasePtsJonq}, with $d_1\ge 4$, $\sum (d_i-1)=k-1\le 3$, so $m=1$, $d_1=4$ and $k=4$. The base points of $\phi$ are thus
\[ t\leftarrow t_1 \leftarrow t_2, \quad p\leftarrow p_{1,1}\leftarrow p_{1,2}\leftarrow p_{1,3}.\]
Since $p_{1,1},p_{1,2},p_{1,3}$ are proximate to $p$, we have $m_{p_{1,1}} + m_{p_{1,2}} + m_{p_{1,3}} \le m_p$. We write $p'\leftarrow p_1'\leftarrow p_2'$ the base points of $\phi'$, and obtain (again by \cref{PropDelta}):
\begin{align*}
 3m_p + m_{p_{1,1}} + m_{p_{1,2}} + m_{p_{1,3}} + m_{t} + m_{t_1} + m_{t_2} &\ge 3d, \\
 m_{p'} + m_1' + m_2' &\ge d.
\end{align*}
Adding the two, and since one of the inequalities is strict, we get
\[
 (m_p + m_{p'} + m_t) + (m_{p} + m_{p_1 '} + m_{t_1}) + (m_{p} + m_{p_2'} + m_{t_2}) + (m_{p_{1,1}} + m_{p_{1,2}} + m_{p_{1,3}}) > 4d.
\]
Since the first sum is greater or equal to the three other ones, we get $m_p + m_{p'} + m_{t} > d$, which gives a contradiction by Bézout's theorem, as $p,p',t\in L$.

\item[\ref{no2coll}]
By~\ref{kprime2}, we have $\deg(\phi')= 2$ and $\deg(\phi) = k \in \{2,3,4\}$. By~\ref{Notaildisjoint}, $\phi$ is as in Case~\ref{BasePtsJonq1} of \cref{Prop:BasePtsJonq}. We assume that $m=1$ for $\phi$ and prove that $\phi$ and $\phi'$ are simple. The base points of $\phi$ are denoted, as in \cref{Prop:BasePtsJonq}\ref{BasePtsJonq1}, by
\[ p \leftarrow p_{1,1} \leftarrow p_{1,2}\leftarrow \cdots \leftarrow p_{1,2k-2}.\]
Proving that $\phi$ is simple corresponds to say that $L$ does not pass through $p_{1,1}$. We will prove it by showing that $m_p+m_{p'}+m_{p_{1,1}}>d$.

As $\deg(\phi')=2$, it also satisfies Case \ref{BasePtsJonq1} of \cref{Prop:BasePtsJonq} (the other cases are of degree $\ge 4$). Its base points are denoted
\[p'\leftarrow p_{1,1}' \leftarrow p_{1,2}'.\]
Proving that $\phi'$ is simple corresponds to say that $L$ does not pass through $p_{1,1}'$. We will prove it by showing that $m_p+m_{p'}+m_{p_{1,1}}'>d$.

As $\deg(\phi(\Lambda))\le d$ and $\deg(\phi'(\Lambda))\le d'$, by \cref{PropDelta} we have:
\begin{align*}m_{p'}+m_{p_{1,1}'}+m_{p_{1,2}'}-d&\ge 0\\
 \sum\nolimits_{i=1}^{k-1} (m_{p}+m_{p_{1,2i-1}}+m_{p_{1,2i}}-d)&\ge 0.\end{align*}
 with one inequality at least being strict.

 The second inequality yields $m_{p}+m_{p_{1,1}}+m_{p_{1,2}}\ge d$.
 As $m_p\ge m_{p_{1,1}}\ge m_{p_{1,2}}$, we obtain $m_p+m_{p_{1,1}}\ge 2d/3$ and $m_p\ge d/3$, with two strict inequalities if $m_{p}+m_{p_{1,1}}+m_{p_{1,2}}>d$.

 Similarly, $m_{p'}+m_{p_{1,1}'}+m_{p_{1,2}'}\ge d$ yields $m_{p'}+m_{p_{1,1}'}\ge 2d/3$ and $m_{p'}\ge d/3$, with two strict inequalities if $m_{p'}+m_{p_{1,1}'}+m_{p_{1,2}'}>d$.

 Summing up, we obtain
 \[(m_p+m_{p_{1,1}})+m_{p'}>2d/3+d/3=d\text{ and }m_p+(m_{p'}+m_{p_{1,1}}')>d/3+2d/3=d,\] as desired.
\qedhere
\end{enumerate}
\end{proof}

\begin{lemma}\label{lemme:tail-and-no-tail}
Suppose $(\phi, \phi', \Lambda)$ is a peak configuration such that $\phi$ has a tail $t\leftarrow t_1\leftarrow t_2$, where $t$ is the maximal base point of $\phi'$ and $m_t(\Lambda)+m_{t_1}(\Lambda)+m_{t_1}(\Lambda)\le \deg(\Lambda)$. Assume that every base point common to $\phi$ and $\phi'$ is contained in $\{t,t_1,t_2\}$. Then, one of the following holds:
\begin{enumerate}
\item
$\phi'$ is simple of degree $2$;
\item\label{tailnotail3}
$\phi'$ is of degree $3$ with two base points in the first neighbourhood of its maximal base point
\parent{Case~\ref{BasePtsJonq1} of \cref{Prop:BasePtsJonq} with $m=2$}.
\end{enumerate}

\end{lemma}
\begin{proof}
We write $d=\deg(\Lambda)$, $k=\deg(\phi)$, $k'=\deg(\phi')$, denote by $m_u=m_u(\Lambda)$ the multiplicity of a point $u$. Let $T=m_{t}+m_{t_1}+m_{t_2}$. Let $p$ and $p'$ be the maximal base points of $\phi$ and $\phi'$ respectively, and let $B$ and $B'$ be the other base points of $\phi$ and $\phi'$. We get
\begin{align*} 0\le d-\deg(\phi(\Lambda))&\stackrel{\ref{PropDelta}}{=}\sum\nolimits_{q\in B} m_q-(d-m_{p})(k-1)\\
0\le d-\deg(\phi'(\Lambda))&\stackrel{\ref{PropDelta}}{=}\sum\nolimits_{q\in B'} m_q-(d-m_{p'})(k'-1).
\end{align*}
Adding these two inequalities, one being strict, and using $m_{t}+m_{t_1}+m_{t_2}\le d$, we obtain
\[ 0<\sum\nolimits_{q\in B\setminus \{t,t_1,t_2\}} m_q+\sum\nolimits_{q\in B'} m_q+m_{p}(k-1)+m_{p'}(k'-1)-d(k+k'-3).\]
By hypothesis, $B\setminus \{t,t_1,t_2\}$ is disjoint from $B'$ and neither $p$ nor $p'$ is contained in these sets, so $\sum\nolimits_{q\in B\setminus \{t,t_1,t_2\}} m_q+\sum\nolimits_{q\in B'} m_q+m_p+m_{p'}\le 3d-3$ (\cref{l:noether_rel}). Hence,
\[m_p(k-2)+m_{p'}(k'-2)>d(k+k'-6)+3.\]
As $m_p+m_{p'}\le d$ (\cref{TwoptsHom}), the left hand side is at most $d(\max(k,k')-2)$, so
\[0>d(k+k'-6)+3-d(\max(k,k')-2)=d(\min(k,k')-4)+3.\]
This gives $\min (k, k') \in\{ 2,3\}$, which implies that $\deg (\phi') \in \{2,3\}$ since $\phi$ has a tail and thus has degree at least $4$ (\cref{Prop:BasePtsJonq}).
We now show the two extra assertions.

Assume first that $\deg (\phi') =2$.
To show that $\phi'$ is simple, we write $t\leftarrow p_{1,1}'\leftarrow p_{1,2}'$ the base points of $\phi'$. 
We need to show that the strict transform of $L$ does not pass through $p_{1,1}'$; this will follow from the inequality \[m_{t} + m_{p_{1,1} '} + m_p > d\] that we prove now.
As $m_t+m_{t_1}+m_{t_2}\le d$, we have $m_p\ge \frac{d}{2}$ (\cref{lemme:low-degree-tail-implies-large-m0}).
The inequality $\deg(\phi'(\Lambda))\le d$ yields $
 m_{t} + m_{p_{1,1}'} + m_{p_{1,2}'} \geq d$ (\cref{PropDelta}). Moreover, one of these two inequalities is strict. Adding them we get
$m_p + m_{t} + m_{p_{1,1}'} + m_{p_{1,2}'} > 3d/2.$
It remains to see that $m_{p_{1,2}'}\le d/2$, which follows from $m_p \geq \frac{d}{2}$ and $m_p + m_{p_{1,2}'} \leq d$ (\cref{TwoptsHom}).

Assume now that $\deg (\phi') =3$.
 Hence, $\phi'$ corresponds to Case~\ref{BasePtsJonq1} of \cref{Prop:BasePtsJonq}. As $\sum_{i=1}^m (d_i-1)=2$, we have either $m=2$ or $m=1$ and $d_1=3$. It remains to show that this latter is impossible. Suppose for contradiction that this holds and write the base points of $\phi'$ as $t \leftarrow p_{1,1} ' \leftarrow p_{1,2}' \leftarrow p_{1,3} ' \leftarrow p_{1,4} '$, where $p_{1,2}'$ proximate to $t$.

As $p_{1,2}'$ proximate to $t$, we have $m_{p_{1,1}'} + m_{p_{1,2}'} \leq m_{t}$. Moreover, $m_{p_{1,2}'} \leq m_{p_{1,1}'}$, so $m_{p_{1,2}'} \le \frac{m_{t}}{2}$.
 Since $m_{p_{1,3}'}, m_{p_{1,4}'} \leq m_{p_{1,2}'}$ we get $m_{p_3'} + m_{p_4'} \leq m_{t}$. The inequality $\deg(\phi'(\Delta))\le d$ yields (\cref{PropDelta}):
\[
\underset{\le m_t}{\underbrace{m_{p_{1,1'}}+m_{p_{1,2}'}}}+\underset{\le m_t}{\underbrace{m_{p_{1,3}'}+m_{p_{1,4}'}}}
\ge 2(d-m_t).\]
This yields $m_t\ge d/2$. We also have $m_p \geq \frac{d}{2}$, and one of these two inequalities is strict, so $m_p + m_{t} > d$ which contradicts \cref{TwoptsHom}.
\end{proof}

\begin{proposition}
\label{prop:induction-peak-configuration}
 Let $(\phi, \phi', \Lambda)$ be a peak configuration.
Assume that $\phi$ and $\phi '$ have different maximal base points $p$ and $q$.

Then, there exist $\psi, \psi'\in \Jonq(3L)$ with maximal base points $p$ and $q$ such that one of the following holds:
\begin{enumerate}[$(\bf A)$]
\item\label{peakind1}
$(\psi, \psi ', \Lambda)$ is a peak configuration satisfying the assumptions of \cref{l:peak_tail} or \cref{l:simple_peak}.
\item\label{peakind2}There exists $\sigma,\sigma' \in \Jonq(3L)$ with the same maximal base point $t\in L\setminus \{p,q\}$ such that $\deg \sigma(\Lambda)<\deg \Lambda$, $\deg \sigma'(\Lambda)<\deg \Lambda$ and such that $(\psi, \sigma, \Lambda)$ and $( \psi ', \sigma',\Lambda)$ are peak configurations satisfying the assumptions of \cref{l:peak_tail}.
\end{enumerate}

Moreover, we may assume that if $\deg \psi(\Lambda)=\deg \Lambda$, then $\deg \phi(\Lambda)=\Lambda$ and $\deg \phi \psi^{-1}$ $<\deg\phi$.
Similarly, if $\deg \psi'(\Lambda)=\deg \Lambda$, then $\deg \phi'(\Lambda)=\Lambda$ and $\deg \phi' \psi'^{-1}< \deg \phi'$.

\end{proposition}

\begin{proof}
By \cref{Cor:SimplificationJonqDec}, there exist $\psi,\psi'\in\Jonq(3L)$ with maximal base points $p$ and $q$ respectively, such that $(\psi,\psi',\Lambda)$ is a peak configuration and such that each of $\psi$ and $\psi'$ has exactly one unique base point in the first neighbourhood of $p$ and $q$ respectively $($Cases~\ref{BasePtsJonq1} or \ref{BasePtsJonq2} of \cref{Prop:BasePtsJonq} with $m=1)$ or exactly two, each one with an even number of infinitely near points $($Case~\ref{BasePtsJonq3} of \cref{Prop:BasePtsJonq} with $m=2)$.
This latter case will actually not occur, but we cannot remove it directly.
In the case of a tail, we may moreover assume that the degree of $\psi$ or $\psi'$ is $4$ if the sum of multiplicities of the tail is at least $\deg(\Lambda)$ (see \cref{Cor:SimplificationJonqDec}\ref{CorSimp2}).
As the base points of $\psi$ and $\psi'$ are contained in those of $\phi$ and $\phi'$ (Corollary~\ref{Cor:SimplificationJonqDec}), we have $\deg \phi \psi^{-1} < \deg \phi$ and $\deg \phi' \psi'^{-1}<\deg \phi'$.
Moreover, by Corollary~\ref{Cor:SimplificationJonqDec}, if $\deg \psi(\Lambda)=\deg \Lambda$ then $\deg \phi(\Lambda)=\Lambda$, and if $\deg \psi'(\Lambda)=\deg \Lambda$ then $\deg \phi'(\Lambda)=\Lambda$.

\textbf{Case 1:}
Suppose first that $\psi, \psi'$ have no common base points.
\cref{l:no_common_points} yields $(\deg (\psi), \deg (\psi'))\in \{(2,2),(2,3),(2,4),(3,2),(4,2)\}$ and implies that both $\psi$ and $\psi'$ satisfy Case~\ref{BasePtsJonq1} of \cref{Prop:BasePtsJonq} (note that the other cases are only in degree $\ge 4$, so this follows from \cref{l:no_common_points}\ref{Notaildisjoint}).
As both $\psi$ and $\psi'$ have exactly one unique base point in the first neighbourhood of $p$ and $q$ respectively, we may apply \cref{l:no_common_points}\ref{no2coll}, and obtain that $\psi$ and $\psi'$ are both simple.
Hence, either $(\psi, \psi ', \Lambda)$ or $(\psi', \psi, \Lambda)$ satisfies the assumptions of \cref{l:simple_peak}, which achieves the proof in this case, providing~\ref{peakind1}.

\textbf{Case 2:} We now consider the case where $\psi,\psi'$ have a common base point.
Exchanging $\psi$ and $\psi'$ if necessary, we may assume that there is a common base point $t\in L\setminus \{p\}$. Hence, $\psi$ has tail $t\leftarrow t_1\leftarrow t_2$, and either $t=q$, or $\psi'$ has a tail containing $t$.

The map $\psi$ corresponds to \cref{Prop:BasePtsJonq}\ref{BasePtsJonq2} with $m=1$. Its base points are thus
\[t\leftarrow t_1\leftarrow t_2, \quad p\leftarrow p_{1,1}\leftarrow p_{1,2}\leftarrow \cdots \leftarrow p_{1,2k-5}\]
with $k=\deg(\psi)\ge 4$. We write $d=\deg(\Lambda)$ and denote by $m_u(\Lambda)=m_u$ the multiplicity of $\Lambda$ at a point $u$.

\textbf{Case 2(i):} $t\not=q$. This implies that $\psi'$ has a tail $t\leftarrow t_1'\leftarrow t_2'$. Exchanging $\psi$ and $\psi'$ if needed, we may assume that $m_{t}+m_{t_1}+m_{t_2}\ge m_{t}+m_{t_1'}+m_{t_2'}$.

We now prove $m_{t}+m_{t_1}+m_{t_2}>d$. Indeed, otherwise \cref{lemme:low-degree-tail-implies-large-m0} implies that $m_p\ge d/2$ and $m_q\ge d/2$ with one inequality being strict, so $m_p+m_q>d$, impossible by \cref{TwoptsHom}.
 The three points $t,t_1,t_2$ are thus not collinear, and there exists $\sigma\in \Jonq(3L)$ quadratic with $t,t_1,t_2$ as base points. The inequality $m_{t}+m_{t_1}+m_{t_2}>d$ yields $\deg \sigma(\Lambda)<\deg \Lambda$ (\cref{PropDelta}). It also implies that $\deg(\psi)=4$, as the sum of multiplicities of the tail is at least the degree.

Assume first that $m_{t}+m_{t_1'}+m_{t_2'}>d$.
The three points $t,t_1',t_2'$ are thus not collinear, and there exists $\sigma'\in \Jonq(3L)$ quadratic with $t,t_1',t_2'$ as base points.
As before we obtain $\deg \sigma'(\Lambda)<\deg \Lambda$ and $\deg(\psi')=4$.
Hence, $(\psi,\sigma,\Lambda)$ and $(\psi',\sigma',\Lambda)$ satisfy the assumptions of \cref{l:peak_tail}, and we get~\ref{peakind2}.
If $m_t + m_{t_1'} + m_{t_2'}\leq d$, then let $\psi''$ be the degree 4 map with the same base points as $\psi '$ except that we replace its tail by $t \leftarrow t_1 \leftarrow t_2$ (its existence follows from \cref{Prop:BasePtsJonq}).
As $m_t+ m_{t_1'} + m_{t_2'}\le d<m_t+m_{t_1}+m_{t_2}$, we have $\deg \psi '' (\Lambda) < \deg \psi ' (\Lambda) \leq \deg (\Lambda)$ (\cref{PropDelta}) and $(\psi '', \sigma, \Lambda)$ satisfies the assumptions of \cref{l:peak_tail}.
We get~\ref{peakind2} with $\sigma ' = \sigma$.

\textbf{Case 2(i):} $t=q$.
Suppose that $m_{t}+m_{t_1}+m_{t_2}>d$. We obtain as before $\deg(\psi)=4$, the three points $t,t_1,t_2$ are not collinear, and there exists $\sigma\in \Jonq(3L)$ quadratic with $t,t_1,t_2$ as base points. We replace $\psi$ with $\sigma$ and get Case~\ref{peakind1}, since $(\psi,\sigma,\Lambda)$ satisfies the assumptions of \cref{l:peak_tail} and $\deg \sigma(\Lambda)<\deg(\Lambda)$. We thus assume that $m_{t}+m_{t_1}+m_{t_2}\le d$, which implies that $m_p\ge \frac{d}{2}$ (\cref{lemme:low-degree-tail-implies-large-m0}).

Suppose that $\psi'$ has a tail $t'\leftarrow t_1'\leftarrow t_2'$ with $p=t'$. As $m_p+m_q\le d$ (\cref{TwoptsHom}), \cref{lemme:low-degree-tail-implies-large-m0} implies that $m_{t'}+m_{t_1'}+m_{t_2'}>d$. We then exchange $\psi$ and $\psi'$ and reduce to the previous paragraph.

The remaining case is when $m_{t}+m_{t_1}+m_{t_2}\le d$, and $\psi'$ has either no tail or a tail not containing $p$. \cref{lemme:tail-and-no-tail} proves that $\psi'$ is a simple quadratic map (the second case is impossible here as we already said that $m=1$ if $\psi'$ is in Case~\ref{BasePtsJonq1} of \cref{Prop:BasePtsJonq}). Denote by $q\leftarrow q_1 \leftarrow q_2$ the base points of $\psi'$. \cref{Prop:BasePtsJonq} yields the existence of $\psi_1\in \Jonq(3L)$ having the same degree and base points as $\psi$, except the tail $t\leftarrow t_1\leftarrow t_2$ replaced by $q\leftarrow q_1 \leftarrow q_2$. As $m_{t}+m_{t_1}+m_{t_2}\le d \le m_{q}+m_{q_1}+m_{q_2}$, we obtain $\deg(\psi_1(\Lambda))\le \deg(\psi(\Lambda))$. Then, \cref{Cor:SimplificationJonqDec} yields the existence of $\psi_2\in \Jonq(3L)$ whose base points are contained in those of $\psi_1$, but being of degree $4$, with again $\deg\psi_2(\Lambda)\le d$, and a strict inequality if $\deg\psi(\Lambda)<d $ or $\deg\psi_1(\Lambda)<d$. We now prove that $\deg\psi_2(\Lambda)< d$ always holds. Indeed, if $\deg\psi(\Lambda)=d$, then $\deg\psi'(\Lambda)<d$, so $m_q+m_{q_1}+m_{q_2}>d$ (\cref{PropDelta}), which implies that $m_q+m_{q_1}+m_{q_2}>m_t+m_{t_1}+m_{t_2}$ and thus $\deg\psi_1(\Lambda)<\deg \psi(\Lambda)=\deg(\Lambda)$.
We obtain that $(\psi_2,\psi,\Lambda)$ satisfies the assumptions of \cref{l:peak_tail} and get~\ref{peakind1}.
\end{proof}

\begin{theorem}
\label{t:lowering_rel}
Every peak configuration $(\phi,\phi',\Lambda)$ fits into a degree lowering relation.
\end{theorem}

\begin{proof}
We prove the result by induction on $(\deg \Lambda, u)$, where $u$ is the degree of the unique $\alpha\in \{\phi,\phi'\}$ satisfying $\deg \alpha(\Lambda)=\deg(\Lambda)$.
If no such $\alpha$ exists, by convention we set $u=0$.

Assume that the result holds for any pair smaller than $(\deg \Lambda, u)$, with respect to the lexicographic order.
 If $\phi, \phi'$ have the same maximal point, then $\phi''=\phi' \phi^{-1} \in \Jonq(3L)$ and this is a relation of length $3$ in $R$. We get a degree lowering relation of length $3$ and are done.
 Otherwise, we apply \cref{prop:induction-peak-configuration}.

In Case~\ref{peakind1}, we obtain the following commutative diagram:
\[
\begin{tikzcd}[dashed]
&& \Lambda \ar[dll," \phi",swap,bend right=7pt]\ar[dl," \psi"]\ar[dr," \psi'",swap] \ar[drr,"\phi'",bend left=7pt]\\
\phi(\Lambda) \ar[r," \theta",swap]&\psi(\Lambda)\ar[rr," \psi'\psi^{-1}",swap]&&\psi'(\Lambda)\ar[r," \theta'^{-1}",swap]& \phi'(\Lambda)
\end{tikzcd}
\]
where $\theta=\psi\phi^{-1},\theta'=\psi'\phi'^{-1}\in \Jonq(3L)$.
Since $(\psi, \psi ', \Lambda)$ is a peak configuration satisfying the assumptions of \cref{l:peak_tail} or \cref{l:simple_peak}, we have $\psi'\psi^{-1}=\chi_n\dots \chi_1$ in $\lb \Jonq(3L) \mid R \rb$, where $\chi_1,\dots,\chi_n\in \Jonq(3L)$ and such that $\deg \chi_i\dots \chi_1\psi(\Lambda)<\deg \Lambda$ for each $i\in \{1,\ldots,n-1\}$.
If $ \psi(\Lambda)$ and $\psi'(\Lambda)$ have both degree strictly smaller than $\Lambda$, this directly gives a degree lowering relation $\phi'\phi^{-1}=\theta'^{-1}\chi_n\dots \chi_1\theta$ for $(\phi,\phi',\Lambda)$, where we replace $\chi_1$ and $\theta$ with their product if $\theta\in \Aut(\P^2)$ is of degree $1$, and similarly with $\theta'^{-1}$ and $\chi_n$ if $\theta'\in \Aut(\P^2)$. The only problematic cases are when $\deg\psi(\Lambda)=\deg(\Lambda)$ or $\deg\psi'(\Lambda)=\deg(\Lambda)$. By symmetry (which exchanges $\phi$ and $\phi'$), we may assume that $\deg\psi(\Lambda)=\deg(\Lambda)$.
Then, \cref{prop:induction-peak-configuration} implies that $\deg \phi(\Lambda)=\deg \Lambda$, $\deg \theta<\deg\phi$ and $\deg \psi'(\Lambda)<\deg \Lambda$.
By induction, the peak configuration $(\theta^{-1},\chi_1,\psi(\Lambda))$ admits a degree lowering relation. Replacing $\chi_1\theta$ with the corresponding product yields a degree lowering relation for $(\phi,\phi',\Lambda)$.

In Case~\ref{peakind2}, we obtain the following commutative diagram:
\[
\begin{tikzcd}[dashed]
&&& \Lambda \ar[dlll," \phi",swap,bend right=13pt]\ar[dll," \psi",swap,bend right=5pt]\ar[dl," \sigma"]\ar[dr," \sigma'",swap] \ar[drr,"\psi'",bend left=5pt]\ar[dr," \sigma'",swap] \ar[drrr,"\phi'",bend left=13pt]\\
\phi(\Lambda)\ar[r," \theta",swap]&\psi(\Lambda) \ar[r," \sigma\psi^{-1}",swap]&\sigma(\Lambda)\ar[rr," \sigma'\sigma^{-1}",swap]&&\sigma'(\Lambda)\ar[r,"\psi'\sigma'^{-1}",swap]& \psi'(\Lambda) \ar[r,"\theta'^{-1}",swap]&\phi'(\Lambda)
\end{tikzcd}
\]
where $\theta=\psi\phi^{-1},\theta'=\psi'\phi'^{-1}\in \Jonq(3L)$.
Since both $(\psi, \sigma, \Lambda)$ and $( \psi ', \sigma',\Lambda)$ are peak configurations satisfying the assumptions of \cref{l:peak_tail}, we may as above write $\sigma\psi^{-1}=\chi_n\dots \chi_1$ in $\lb \Jonq(3L) \mid R \rb$ where $\deg \chi_i\dots \chi_1\phi(\Lambda)<\deg \Lambda$ for each $i$, and similarly for $\psi'\sigma'^{-1}=\chi_m'\dots \chi_1'$.
As $\sigma'\sigma^{-1}\in \Jonq(3L)$ and $ \sigma(\Lambda)$, $\sigma'(\Lambda)$ both have degree smaller than $\deg \Lambda$, we obtain a degree lowering relation if $\deg \psi(\Lambda)<\deg \Lambda$ and $\deg \psi'(\Lambda)<\deg \Lambda$. As before, it suffices then to study the case where $\deg \psi(\Lambda)=\deg \Lambda$, $\deg \phi(\Lambda)=\deg \Lambda$, $\deg \theta<\deg \phi$ and $\deg \psi'(\Lambda)<\deg \Lambda$.
Again we conclude by applying the induction hypothesis to $(\theta^{-1},\chi_1,\psi(\Lambda))$.
\end{proof}

\subsection{Relations in \texorpdfstring{$\Bir(\P^2, 3L)$}{Bir(P2,3L)}}
\begin{definition}\label{dfn:}
Let $r\ge 1$ and let $\phi_1, \dots, \phi_r \in \Jonq(3L)$. We say that the product
 \[
 \phi_r \dots \phi_1
 \]
 is in \emph{normal form} if, writing $d_i = \deg (\phi_i \dots \phi_1)$, the sequence
 \[
 (d_1,\ldots,d_r)
 \]
 is strictly increasing.
\end{definition}

\begin{theorem}
\label{t:normal_form}
Let $\phi_1, \dots, \phi_r \in \Jonq(3L)$.
Then there exists a word $\psi_s \dots \psi_1$ in normal form such that $\phi_r \dots \phi_1 = \psi_s \dots \psi_1$ in the group $\langle \Jonq(3L) \mid R \rangle$. 
\end{theorem}

\begin{proof}
Let $W = \phi_r \dots \phi_1$ be a word, with $\phi_1,\ldots,\phi_r\in \Jonq(3L)$. 
We write $d_i = \deg (\phi_i \dots \phi_1)$ for $i \in \{0 ,\dots, r\}$, where the case $i = 0$ corresponds to the identity map, so $d_0=1$.
Let $I$ be the set of indices $i$ in $\left\{ 1, \dots,r-1 \right\}$ such that $d_{i-1} \leq d_i$ and $d_{i+1} \leq d_i$ with one of the inequalities being strict. We then define
\[
D(W) = \left( d_{\max} (W), t_{\max}(W), r \right),
\]
where $d_{\max}(W) = \max_{i \in I} d_i$, 
$t_{\max} (W) = \# \left\{ i \mid 1 \le i \le r, d_i = d_{\max} (W) \right\}$.

The element $D(W)$ belongs to the set $(\Z_{\geq 0})^3$, which is well ordered with the lexicographic order.
We now prove the result by induction on $D(W)$.
We thus assume that $W$ is not in normal form, and that the result holds for all words smaller than $W$.

If $\phi_i\in A$ for some $i\in \{1,\ldots,r\}$, then  $r\ge 2$ (if $r=1$, then $\phi_1$ is in normal form), we replace $\phi_i\phi_{i-1}$ or $\phi_{i+1}\phi_i$ with its product. This relation holds in $\langle \Jonq(3L) \mid R \rangle$ and this gives a new word $W'$ with smaller length and $(d_{\max}(W'),t_{\max}(W'))\le (d_{\max}(W),t_{\max}(W))$. The result holds by induction.

Otherwise, let $i_0 \in I$ be an index such that $d_{i_0} = d_{\max} (W)$.
Writing $\Lambda_0$ the linear system of $(\phi_{i_0} \dots \phi_{1})^{-1}$, the triple $(\Lambda_0, \phi_{i_0 + 1}, \phi_{i_0}^{-1})$ is a peak configuration.
By \cref{t:lowering_rel}, there exist $\psi_1, \dots, \psi_s \in \Jonq(3L)$ such that
\[
 \phi_{i_0 + 1}\phi_{i_0} = \psi_s \dots \psi_1
\]
in $\langle \Jonq (3L) \mid R \rangle$ and such that, for each $t \in \{1,\dots, s-1\}$, we have 
\[
 \deg (\psi_t \dots \psi_1 \phi_{i_0 -1} \dots \phi_1) < d_{\max} (W).
\]
Note that $\deg (\psi_s \dots \psi_1 \phi_{i_0 -1} \dots \phi_1) = \deg(\phi_{i_0 +1}\phi_{i_0} \dots \phi_1 ) \leq d_{\max} (W).$ 
Let $W'$ be the new word
\[
 W' = \phi_r \dots \phi_{i_0 + 2} \psi_s \dots \psi_1 \dots \phi_{i_0 -1} \dots \phi_1.
\]
We either get $d_{\max}(W') < d_{\max} (W)$, or we have equality but then $t_{\max}(W') < t_{\max} (W)$.
In any case, $D(W') < D(W)$. Hence, the result holds by induction.
\end{proof}

We can now finish give the proof of \cref{t:main_bis}:

\begin{proof}[Proof of Theorem~\ref{t:main_bis}]
By \cref{t:generators}, $\Jonq(3L)$ is a generating set of $\Bir(\P^2, 3L)$.
So we have a surjective morphism $\langle \Jonq (3L) \mid R \rangle \twoheadrightarrow \Bir(\P^2, 3L)$.
To show that this is an isomorphism, we take $\phi_1, \dots, \phi_r \in \Jonq( 3L)$ such that $\phi_r \dots \phi_1 = \id$ in $\Bir(\P^2, 3L)$, and prove that $\phi_r \dots \phi_1 = \id$ in $\langle \Jonq(3L) \mid R \rangle$. \cref{t:normal_form} yields the existence of $s\ge 1$ and $\psi_1,\ldots,\psi_s\in \Jonq(3L)$ such that
 \[
 \id = \phi_r \dots \phi_1 = \psi_s \dots \psi_1
 \]
 in $\langle \Jonq(3L) \mid R \rangle$, where $\psi_s \dots \psi_1$ is in normal form.
As $\deg(\psi_s \dots \psi_1)=1$, this implies that $s=1$ and $\psi_1=\mathrm{id}$.
\end{proof}

\section{Alternative formulation of the main theorem}
\label{sec:reformulation}
\subsection{Relations of length 3 in \texorpdfstring{$\Jonq(3L)$}{Jonq(3L)}}

\begin{lemma}\label{l:jonq_relation}Let $\phi_1, \phi_2, \phi_3 \in \Jonq(3L)$ such that $\phi_3 \phi_2 \phi_1 = 1$ in $\Bir(\P^2,3L)$.
Then there are some points $p_1, p_2,p_3, p_4 = p_1 \in \P^2$ such that $\phi_i$ sends the pencil of lines through $p_{i}$ to the pencil of lines through $p_{i+1}$ for each $i\in \{1,2,3\}$.

In particular, if $a_1, a_2, a_3 \in A$ are such that $a_i(\pcoor{1:0:0}) = p_i$, we get
\[
\id =
(\underbrace{a_1^{-1} \phi_3 a_3}_{b_3\in B})
(\underbrace{a_3^{-1} \phi_2 a_2}_{b_2\in B})
(\underbrace{a_2^{-1} \phi_1 a_1}_{b_1\in B}).
\]
\end{lemma}

\begin{proof}
We may replace $(\phi_1,\phi_2,\phi_3)$ with $(\phi_2,\phi_3,\phi_1)$ or $(\phi_3,\phi_1,\phi_2)$ if needed.
We need to find a pencil $\Lambda$ of lines of $\P^2$ such that $\varphi_1(\Lambda)$ and $\varphi_2(\varphi_1(\Lambda))$ are pencils of lines.
If some $\phi_i\in A$ is linear, we may assume that $\varphi_2\in A$, and choose $\Lambda$ to be a pencil of lines such that $\varphi_1(\Lambda)$ is a pencil of lines.
We may thus assume that none of $\varphi_1,\varphi_2,\varphi_3$ belongs to~$A$.
For each $i\in \{1,2,3\}$, we write $d_i=\deg(\varphi_i)\ge 2$, denote by $p_i$ the maximal base point of $\varphi_i$ and by $q_i$ the maximal base point of $(\varphi_i)^{-1}$, so that $\varphi_i$ sends the pencil $\Lambda_{p_i}$ of lines through $p_i$ onto the the pencil $\Lambda_{q_i}$ of lines through $q_i$.

If $q_1=p_2$, we can choose $\Lambda=\Lambda_{p_1}$ and obtain $\varphi_1(\Lambda_{p_1})=\Lambda_{q_1}$ and $\varphi_2(\varphi_1(\Lambda_{p_1}))=\Lambda_{q_2}$, which achieves the proof. The same holds if $q_2=p_3$ or $q_3=p_1$.

We may thus assume that $q_1\not=p_2$, $q_2\not=p_3$ and $q_3\not=p_1$ and derive a contradiction.

We may also assume that $d_2=\max(d_1,d_2,d_3)$.
We now prove that $d_2\ge 4$.
Indeed, if $d_2\le 3$, then $\varphi_1$ and $\varphi_2^{-1}$ both have degree $2$ or $3$ and thus have only one proper base point (\cref{Prop:BasePtsJonq}), namely $q_1$ and $p_2$ respectively.
As $q_1\not=p_2$, the two maps have no common base point and thus $\deg(\varphi_3)=\deg(\varphi_2\varphi_1)=\deg(\varphi_2)\deg(\varphi_1)\ge 4$, impossible.

Let us write $\Omega=\phi_2(\Lambda_{q_1})$ and $r=\deg(\Omega)$.
By \cref{PropDelta}, we have $r=d_2-m$, where $m\in \{0,1\}$ is the multiplicity of $q_1$ as a base point of $\phi_2$, so $r\ge 3$.

As $p_2\not=q_1$, the two pencils $\Lambda_{p_2}$ and $\Lambda_{q_1}$  give a birational map $\P^2\dasharrow \p^1\times \p^1$. This implies that $\Lambda_{q_2}=\phi_2(\Lambda_{p_2})$ and $\Omega=\phi_2(\Lambda_{q_1})$ also define a birational map $\P^2\dasharrow \p^1\times \p^1$. As $\deg( \Omega)\ge 3$, the multiplicity of $\Omega$ at $q_2$ is $r-1$.

We now use $\Omega=\phi_2(\Lambda_{q_1})=\phi_2\phi_1(\Lambda_{p_1})=\phi_3^{-1}(\Lambda_{p_1})$ and apply the same argument to $\phi_3^{-1}$: As $\Lambda_{p_1}$ and $\Lambda_{q_3}$ define a birational map to $\p^1\times\p^1$, so do $\Omega=\phi_3^{-1}(\Lambda_{p_1})$ and $\Lambda_{p_3}=\phi_3^{-1}(\Lambda_{q_3})$. Hence, the multiplicity of $\Omega$ at $p_3$ is $r-1$.

We obtain two proper points $q_2\not=p_3$ which have both multiplicity $r-1$, so by Bézout's theorem $2(r-1)\le r$, which contradicts $r\ge 3$. 
\end{proof}

Let $G$ be a group, and $G_1, G_2 \subseteq G$ two subgroups such that $\langle G_1, G_2 \rangle = G$.
Consider the abstract free product $G_1 * G_2$, and denote by $i_1\colon G_1 \to G_1 * G_2$, $i_2\colon G_2 \to G_1 * G_2$ the two natural injections.
We denote by $N$ the smallest normal subgroup in $G_1*G_2$ containing all elements of the form $i_1(h)i_2(h)^{-1}$, with $h \in G_1 \cap G_2$.
There is a natural surjective homomorphism $\phi\colon (G_1*G_2)/N \to G$.
We say that $G$ is the \emph{amalgamated product} of its subgroups $G_1$ and $G_2$ along their intersection, denoted by $G = G_1 *_{G_1\cap G_2} G_2$, if $\phi$ is an isomorphism.
Informally, this means that all relations in $G = \langle G_1, G_2
\rangle$ come from the relations inside $G_1$ or $G_2$.
The amalgamated product satisfies the following universal property.
If $H$ is a group and $\phi_1\colon G_1 \to H$, $\phi_2\colon G_2 \to H$ are two homomorphisms that coincide on $G_1 \cap G_2$, then there is a unique homomorphism $\phi\colon G_1 *_{G_1 \cap G_2} G_2 \to H$ that extends $\phi_1$ and $\phi_2$.

\begin{lemma}
\label{l:amalgam_free}
Let $G = G_1 *_{G_1 \cap G_2} G_2$ be an amalgamated product.
Let
\[
X = \{ g_1 g_2 g_1' \in G \mid g_1, g_1' \in G_1, g_2 \in G_2 \}.
\]
Let $F(X)$ be the free group over $X$, and let $R \subseteq F(X)$ be the set of words of the form $xyz$, with $x,y,z \in X$, and $xyz = 1$ in $G$.
Then $G = \lb X \mid R \rb$.
\end{lemma}

\begin{proof}
Since $G_1 \cup G_2 \subseteq X$, and any group  is the quotient of the free group over all its elements modulo relations of length 3, there are natural homomorphisms $\phi_i\colon G_i \to F(X)/\lbb R \rbb$, which coincide on $G_1 \cap G_2$.
By the universal property of the amalgamated product described above, these homomorphisms admit a common extension $\phi\colon G \to F(X)/\lbb R \rbb$.

We claim that $\phi$ is surjective.
To distinguish elements from $G$ and $F(X)$, we denote by $[x]$ the word in $F(X)$ corresponding to an element $x\in X\subseteq G$. We prove that $\varphi(x)=[x]$ for each $x\in X$.
There exist $g_1,g_1'\in G_1$, $g_2\in G_2$ such that $x=g_1g_2g_1'$. As
we have the two relations of length three
\(
 [g_2] [g_1'] = [g_2 g_1']
\)
and
\( [g_1] [g_2 g_1 '] = [g_1 g_2 g_1 ']\),
we obtain
\[
 \phi(x)=\phi(g_1 g_2 g_1') = [g_1] [g_2] [g_1 '] = [g_1] [g_2 g_1 '] = [g_1 g_2 g_1 ']=  [x] .
\]
Thus $X \subseteq \phi (G)$, and therefore $\phi$ is surjective, as claimed.

Conversely, since $X \subseteq G$, there is a natural homomorphism from the free group $F(X)$ to $G$, and each element in $R$ is in the kernel.
So we get a homomorphism $\psi\colon F(X)/\lbb R \rbb \to G$,
which is surjective since $G_1\cup G_2\subseteq X$ generate $G$.
As $\psi(\phi(t))=\psi([t])=t$ for each $t\in G_1\cup G_2$, we obtain $\psi\circ \phi=\mathrm{id}_G$. This implies that $\phi$ is injective and thus bijective.
\end{proof}

\begin{lemma}
\label{lemm:Jon3LAB}
There is a natural isomorphism
\[
\lb \Jonq(3L) \mid R_3 \rb \iso A *_{A \cap B} B,
\]
where $R_3 \subseteq F(\Jonq(3L))$ is the set of all words of the form $\phi_3\phi_2\phi_1$, where $\phi_1,\phi_2,\phi_3\in \Jonq(3L)$ and $\phi_3\phi_2\phi_1=\id$ in $\Bir(\P^2, 3L)$.
\end{lemma}

\begin{proof}
By \cref{l:amalgam_free}, writing each element of $\Jonq(3L)$ as $aba'$, with $a,a'\in A$ and $b\in B$, we obtain an isomorphism
$
\psi\colon \lb \Jonq(3L) \mid R_3' \rb\to A *_{A \cap B} B,
$
where $R_3'\subseteq F(\Jonq(3L))$ consists of all
words of the form $xyz$ with $x,y,z\in \Jonq(3L)$ and where $\psi(x)\psi(y)\psi(z)=1$ in $A *_{A \cap B} B$.

We now consider $\phi_1,\phi_2,\phi_3\in \Jonq(3L)$ such that $\phi_3 \phi_2 \phi_1=\id$ in $\Bir(\P^2, 3L)$, and we want to prove that $\psi(\phi_3)\psi(\phi_2)\psi(\phi_1)=1$ in $A *_{A \cap B} B$.
By~\cref{l:jonq_relation}, there exist $a_1, a_2, a_3 \in A$ such that
\[
\id =
(\underbrace{a_1^{-1} \phi_3 a_3}_{b_3\in B})
(\underbrace{a_3^{-1} \phi_2 a_2}_{b_2\in B})
(\underbrace{a_2^{-1} \phi_1 a_1}_{b_1\in B}).
\]
By definition of $\psi$, we obtain $\psi(\phi_3)=a_1b_3a_3^{-1}$, $\psi(\phi_2)=a_3b_2a_2^{-1}$ and $\psi(\phi_1)=a_2b_1a_1^{-1}$.
As $b_3b_2b_1 = 1$ in $B$, the relation $\psi(\phi_3)\psi(\phi_2)\psi(\phi_1)=(a_1b_3a_3^{-1})(a_3b_2a_2^{-1})(a_2b_1a_1^{-1})=1$ holds in the amalgamated product $A *_{A \cap B} B$, as desired.
\end{proof}

\begin{proof}[{Proof of \cref{t:main}}]
The isomorphism from \cref{lemm:Jon3LAB} sends 
$
(-xy^2+y, \tfrac1y ) \circ (x, -y+ \tfrac1x)
$
to the composition
$
\bigl(-xy^2+y, \tfrac1y \bigr) \circ (y,x) \circ \bigl(-x+ \tfrac1y,y \bigr) \circ (y,x),
$
so we get an  isomorphism 
\[
\Bir(\P^2, 3L)\stackrel{\text{Theorem~\ref{t:main_bis}}}{=}\langle \Jonq(3L) \mid R \rangle \iso
(A *_{A \cap B} B) / N,
\]
where $N$ is the normal closure of $( \bigl(-xy^2+y, \tfrac1y \bigr) \circ (y,x) \circ \bigl(-x+ \tfrac1y,y \bigr) \circ (y,x) )^2.$
\end{proof}

\bibliographystyle{myalpha2024}
\bibliography{biblio_main}

\end{document}